\documentclass[12 pt]{article}
\usepackage{amsmath,amsfonts,amsthm,amssymb,mathrsfs,latexsym,paralist,xy}
\usepackage{tikz, soul}
\usepackage{lipsum}

\newtheorem{thm}{Theorem}[section]
\newtheorem{prop}[thm]{Proposition}

\newtheorem{lemma}[thm]{Lemma}
\theoremstyle{remark}
\newtheorem{rmk}[thm]{Remark}
\newtheorem{example}[thm]{Example}
\theoremstyle{definition}
\newtheorem{defn}[thm]{Definition}
\newtheorem{proposition}[thm]{Proposition}

\newcommand{\z}{^{0}}

\renewcommand{\sp}{\textrm{span}}

\newcommand{\bi}{\begin{itemize}}
\newcommand{\ei}{\end{itemize}}
\newcommand{\be}{\begin{enumerate}}
\newcommand{\ee}{\end{enumerate}}

\newcommand{\C}{\mathbb{C}}

\newcommand{\T}{\mathbb{T}}
\renewcommand{\H}{\mathcal{H}}

\newcommand{\R}{\mathbb{R}}
\newcommand{\N}{\mathbb{N}}
\newcommand{\Z}{\mathbb{Z}}

\renewcommand{\l}{\langle}
\newcommand{\clsp}{\operatorname{\overline{span}}}

\renewcommand{\r}{\rangle}

\providecommand{\keywords}[1]{{\textit{Key words and phrases:}} #1}
\providecommand{\classification}[1]{{\textit{2010 Mathematics Subject Classification:}} #1}
\usepackage[margin=1in]{geometry}

\AtEndDocument{\bigskip{\footnotesize%
  \textsc{Carla Farsi, Elizabeth Gillaspy, Sooran Kang, Judith Packer : Department of Mathematics, University of Colorado at Boulder, Boulder, Colorado, 80309-0395} \par
  \textit{E-mail address}: \texttt{farsi@euclid.colorado.edu, Elizabeth.Gillaspy@colorado.edu, sooran@colorado.edu, packer@euclid.colorado.edu} \par
}}

\begin{document}

\title{Wavelets and graph $C^*$-algebras}
\author{Carla Farsi, Elizabeth Gillaspy, Sooran Kang and Judith Packer}

\maketitle

\begin{abstract} Here we give an overview on the connection between wavelet theory and representation theory for graph $C^{\ast}$-algebras, including the higher-rank graph $C^*$-algebras of A. Kumjian and D. Pask. Many authors have studied different aspects of this connection over the last 20 years, and we begin this paper with a survey of the known results.  We then discuss several new ways to generalize these results and obtain wavelets associated to representations of higher-rank graphs. In \cite{FGKP}, we introduced the ``cubical wavelets" associated to a higher-rank graph. Here, we generalize this construction to build wavelets of arbitrary shapes. We also present a different but related construction of wavelets associated to a higher-rank graph, which we anticipate will have applications to traffic analysis on networks.  Finally, we generalize the spectral graph wavelets of \cite{hammond} to higher-rank graphs, giving a third family of wavelets associated to higher-rank graphs.
\end{abstract}

\classification{46L05, 42C40}


\tableofcontents

\section{Introduction}
Wavelets were developed by S.  Mallat, Y. Meyer, and I. Daubechies in the late 1980's \cite{mallat} and early 1990's as functions on $L^2(\mathbb R^n)$ that were well-localized in either the ``time" or ``frequency" domain, and thus could be used to form an orthonormal basis for $L^2(\mathbb R^n)$ that behaved well under compression algorithms, for the purpose of signal or image storage.  Mallat and  Meyer developed a very important algorithm, the so-called multiresolution analysis algorithm, as a way to construct so-called ``father wavelets" and ``mother wavelets" on $L^2(\mathbb R)$ from associated ``filter functions" \cite{mallat}, \cite{Strichartz1}.

 Beginning with the initial work of O. Bratteli and P. Jorgensen in the mid 1990's, which gave a relationship between multiresolution analyses for wavelets on $L^2(\mathbb R)$ and certain types of representations of the Cuntz algebra ${\mathcal O}_N,$  the representations of certain graph $C^*$-algebras and the constructions of wavelets on $L^2(\mathbb R)$ were shown to be related. To be more precise, in 
1996, O. Bratteli and P. Jorgensen first announced in \cite{BJA} that there was a correspondence between dilation-translation wavelets of scale $N$ on $L^2(\mathbb R)$ constructed via the multiresolution analyses of Mallat and  Meyer, and certain representations of the Cuntz algebra ${\mathcal O}_N$.  Later, together with D. Dutkay, Jorgensen extended this analysis to describe wavelets on $L^2$-spaces corresponding to certain inflated fractal sets (\cite{dutkay-jorgensen-fractals}) constructed from iterated function systems.  The material used to form these wavelets also gave rise to representations of ${\mathcal O}_N.$ 
 Recently, in \cite{dutkay-jorgensen-monic}, Dutkay and Jorgensen were able to relate representations of the Cuntz algebra  of ${\mathcal O}_N$ that they termed ``monic" to representations on $L^2$ spaces of other non-Euclidean spaces  carrying more locally defined branching operations related to dilations. The form that monic representations take has similarities to earlier representations of ${\mathcal O}_N$ coming from classical wavelet theory. 

Initially, the wavelet function or functions were made into an orthonormal basis by applying translation and dilation operators to a fixed family of functions, even in Dutkay and Jorgensen's inflated fractal space setting.  However,  already the term ``wavelet" had come  to have a broader meaning as being a function or finite collection of functions on a measure space $(X,\mu)$ that could be used to construct either an orthonormal basis or frame basis of $L^2(X,\mu)$ by means of operators connected to algebraic or geometric information relating to $(X,\mu).$

In 1996, A. Jonsson described collections of functions on certain finite fractal spaces that he defined as wavelets, with the motivating example being Haar wavelets restricted to the Cantor set \cite{jonsson}.  Indeed, Jonsson had been inspired by the fact that the Haar wavelets, which are discontinuous on $[0,1],$ are in fact continuous when restricted to the fractal Cantor set, and therefore can be viewed as a very well-behaved orthonormal basis giving a great deal of information about the topological structure of the fractal involved. Also in 1996, just slightly before Jonsson's work,  R. Strichartz analyzed wavelets on Sierpinski gasket fractals in \cite{Strichartz2}, and noted that since fractals built up from affine iterated function systems such as the Sierpinski gasket fractal had locally defined translations, isometries and dilations, they were good candidates for an orthonormal basis of wavelets.  Strichartz's wavelets were defined by constructing an orthonormal basis from functions that at each stage of the iteration building the fractal  had certain properties (such as local constance) holding in a piecewise fashion (\cite{Strichartz2}).

 With Jonsson's and Strichartz's constructions in mind, but starting from an operator-algebraic viewpoint, in 2011, M. Marcolli and A. Paolucci looked at representations of the Cuntz-Krieger $C^{\ast} $-algebras $\mathcal{O}_A$ on certain $L^2$-spaces, and showed that one could construct generalized ``wavelet" families, by using the isometries and partial isometries naturally generating the $C^{\ast}$-algebras $\mathcal{O}_A$ to operate on the zero-order and first-order scaling functions and wavelet functions, thus providing  the orthonormal basis for the Hilbert space in question.  The Cuntz-Krieger $C^{\ast}$-algebras $\mathcal{O}_A$ are $C^{\ast}$-algebras generated by partial isometries, where the relations between the isometries are determined by the matrix $A$; interpreting the matrix $A$ as the adjacency matrix of a graph allows us to view the Cuntz-Krieger $C^*$-algebras as graph algebras.  %associated to finite directed graphs that are generated by partial isometries.  
 Thus, in the wavelet constructions of Marcolli and Paolucci, the partial isometries coming from the graph algebra act in a sense similar to the localized dilations and isometries observed by Strichartz in \cite{Strichartz2}. More precisely, by showing that it was possible to represent certain Cuntz-Krieger $C^{\ast}$-algebras on  $L^2$-spaces associated to %certain
  non-inflated fractal spaces, Marcolli and  Paolucci related the work of Bratteli and Jorgensen and Dutkay and Jorgensen to the works of Jonsson and Strichartz. Moreover, they showed that in this setting, certain families related to Jonsson's wavelets could be constructed by acting on the so-called scaling functions and wavelets by %certain
  partial isometries geometrically related to the directed graph in question.

In this paper, in addition to giving a broad overview of this area, we will discuss several new ways to generalize these results and obtain wavelets associated to representations of  directed graphs and higher-rank graphs.   For a given directed graph $E$, the graph $C^*$-algebra $C^*(E)$ is the universal $C^*$-algebra generated by a collection of projections associated to the vertices and partial isometries associated to the edges that satisfy certain relations, called the Cuntz-Krieger relations. It has been shown that graph $C^*$-algebras not only generalize Cuntz-Krieger algebras, % up to Morita equivalence
 but they also  include (up to Morita equivalence) a fairly wide class of $C^*$-algebras such as the AF-algebras, Kirchberg algebras with free $K_1$-group, and various noncommutative algebras of functions on quantum spaces. One of the benefits of studying graph $C^*$-algebras is that  very abstract properties of $C^*$-algebras can be visualized via concrete characteristics of underlying graphs.  See the details in the book ``Graph Algebras''  by Iain Raeburn \cite{Raeburn} and the references therein.

Higher-rank graphs, also called $k$-graphs, were introduced in \cite{KP} by Kumjian and Pask as higher-dimensional analogues of directed graphs, and they provide a combinatorial model to study the higher dimensional Cuntz-Krieger algebras of Robertson and Steger \cite{RS-London, RS}. Since then, $k$-graph $C^*$-algebras have been studied by many authors and have provided many examples of various classifiable $C^*$-algebras, and the study of fine structures and invariants of $k$-graph $C^*$-algebras can be found in \cite{E, PRRS, RSY1, robertson-sims, RoS2, S, KangPask, ckss, SZ, SP}. Also $k$-graph $C^*$-algebras provide many examples of non-self-adjoint algebras and examples of crossed products. (See  \cite{DPY, DY, KrP, Po, BR, Ex, FPS}).  Recently, twisted $k$-graph $C^*$-algebras have been developed in \cite{KPS-hlogy, KPS-twisted, KPS-twisted-BD, KPS-twisted-K-theory, SWW}; these provide many important examples of $C^*$-algebras including noncommutative tori.
Moreover, specific examples of dynamical systems on $k$-graph $C^*$-algebras %with specific examples 
have been studied in \cite{paskraebuweaver-periodic, paskraebuweaver-family}, and the study of KMS states  with gauge dynamics can be found in \cite{aHLRS1, aHLRS2, aHLRS3, aHLRS4, aHKR}. Furthermore, the works in \cite{PR, PRS} show that $k$-graph $C^*$-algebras can be realized as noncommutative manifolds and have the potential to enrich the study of noncommutative geometry.

 The first examples of directed graph algebras are the Cuntz $C^{\ast}$-algebras ${\mathcal O}_N$ defined for any integer $N\geq 2$, which are generated by  $N$ isometries satisfying some elementary relations.  In the late 1990's it was realized by Bratteli and Jorgensen (\cite{BJA}, \cite{BJ1}) that the theory of multiresolution analyses for wavelets and the theory of certain representations of ${\mathcal O}_N$ could be connected through filter functions, or quadrature mirror filters, as they are sometimes called.  We review this relationship in Section 2, since this was the first historical connection between wavelets and $C^{\ast}$-algebras.  In this section, we also relate the filter functions associated to fractals coming from affine iterated function systems, as first defined by Dutkay and Jorgensen in \cite{dutkay-jorgensen-fractals}, as well as certain kinds of representations of ${\mathcal O}_N$ defined by Bratteli and Jorgensen called {\it monic} representations, as all three of these representations of ${\mathcal O}_N$ (those coming from \cite{BJ1}, from \cite{dutkay-jorgensen-fractals}, and  from \cite{dutkay-jorgensen-monic}) correspond to what we call a {\it Cuntz-like} family of functions on $\mathbb T.$ Certain forms of monic representations, when moved to $L^2$-spaces of Cantor sets % of the form $K_N$, 
associated to $\mathcal{O}_N$,
 can viewed as examples of semibranching function systems, and thus are precursors of the types of representations of Cuntz-Krieger algebras  %giving rise to wavelets 
 studied by Marcolli and Paolucci in \cite{MP}.  In Section 3 we give an overview of the work of Marcolli and Paolucci from \cite{MP}, discussing semibranching function systems satisfying a Cuntz-Krieger condition and the representations of Cuntz-Krieger $C^{\ast}$-algebras on the $L^2$-spaces of fractals.  We state the main theorem of Marcolli and Paolucci from \cite{MP} on the construction of wavelets on these spaces, which generalizes the constructions of Jonsson and Strichartz, but we omit the proof of their theorem.  However, we give the proof that, given any Markov probability measure on the fractal space $K_N$ associated to $\mathcal{O}_N$, there exists an associated representation of $\mathcal{O}_N$ and a family of related wavelets.  %of the existence of wavelets for a representation of ${\mathcal O}_N$ for any Markov probability measure on $K_N.$ 
 In Section 4, we review the definition of directed graph algebras and also review $C^{\ast}$-algebras associated to finite higher-rank graphs (first defined by Kumjian and Pask in  \cite{KP}) and then generalize the notion of semibranching function systems to higher-rank graph algebras via the definition of $\Lambda$-semibranching function systems, first introduced in \cite{FGKP}.  In Section 5, we use the representations arising from $\Lambda$-semibranching function systems to construct wavelets of an arbitrary rectangular shape on the $L^2$-space of the infinite path space $\Lambda^{\infty}$ of any finite strongly connected $k$-graph $\Lambda.$  In so doing we generalize a main theorem from \cite{FGKP} and answer in the affirmative a question posed to one of us by Aidan Sims. In Section 6, motivated by work of Marcolli and Paolucci for wavelets associated to Cuntz-Krieger $C^{\ast}$-algebras, we discuss the use of  $k$-graph wavelets in the construction of (finite-dimensional) families of wavelets that can hopefully be used in traffic analysis on networks, and also discuss generalizations of wavelets on the vertex space of a $k$-graph that can be viewed as eigenvectors of the (discrete) Laplacian on this vertex space.  We analyze the wavelets and the wavelet transform in this case, thereby generalizing some results of Hammond, Vanderghynst, and Gribonval from \cite{hvrg}.
 
 This work was partially supported by a grant from the Simons Foundation (\#316981 to Judith Packer).

%{\color{red}  Should we finish up the intro by mentioning how we eventually want to connect to work of Antoine Julien et al on eigenvectors of the generalized Laplacian?}

\section{$C^{\ast}$-algebras and work by Bratteli and  Jorgensen and Dutkay and Jorgensen on representations of ${\mathcal O}_N$}

We begin by giving a very brief overview of $C^{\ast}$-algebras and several important constructions in $C^{\ast}$-algebras that will prove important in what follows.  Readers interested in further detail can examine B. Blackadar's book \cite{black} (to give just one reference).

\begin{defn}  A $C^{\ast}$-algebra is a Banach algebra which we shall denote by ${\mathcal A}$ that has assigned to it an involution $\ast$ such that the norm of ${\mathcal A}$ satisfies the so-called  $C^*$-identity:
$$\|a^{\ast}a\|\;=\;\|a\|^{2},\;\forall a\in {\mathcal A}.$$  By a celebrated theorem of I. Gelfand and M. Naimark, every $C^{\ast}$-algebra can be  represented faithfully as a Banach $\ast$-subalgebra of the algebra of all bounded operators on a Hilbert space.
\end{defn}
$C^{\ast}$-algebras have a variety of important and very useful applications in mathematics and physics. $C^{\ast}$-algebras can be used to study the structure of topological spaces, as well as the algebraic and representation-theoretic structure of locally compact topological groups. Indeed, $C^{\ast}$-algebras provide one framework for a mathematical theory of quantum mechanics, with observables and states being described precisely in terms of self-adjoint operators and mathematical states on $C^{\ast}$-algebras.  When there are also symmetry groups involved in the physical system, the theory of $C^{\ast}$-algebras allows these symmetries to be incorporated into the theoretical framework as well.

In this paper, we will mainly be concerned with $C^{\ast}$-algebras constructed from various relations arising from directed graphs and higher-rank graphs.  These are combinatorial objects satisfying certain algebraic relations that are most easily represented by projections and partial isometries acting on a Hilbert space.  The $C^{\ast}$-algebras that we will study will also contain within them certain naturally defined commutative $C^{\ast}$-algebras,  as fixed points of a canonical gauge action. These commutative $C^{\ast}$-algebras can be realized as continuous functions on Cantor sets of various types, and under appropriate conditions there are measures on the Cantor sets, and associated representations of the $C^{\ast}$-algebras being studied on the $L^2$-spaces of the Cantor sets. These will be the representations that we shall study, but we will first briefly review the notion of $C^{\ast}$-algebras characterized by universal properties.  %Our main reference for this is the book by B. Blackadar \cite{black}.

\begin{defn}  (\cite{black})
Let ${\mathcal G}$ be a (countable) set of generators, closed under an involution $\ast,$ and ${\mathcal R}({\mathcal G})$ a set of algebraic relations on the elements of ${\mathcal G},$ which have as a restriction that it must be possible to realize ${\mathcal R}({\mathcal G})$ among operators on a Hilbert space $\mathcal{H}$. It is also required that ${\mathcal R}({\mathcal G})$ must place an upper bound on the norm of each generator when realized as an operator.
A \textbf{representation} $(\pi, {\mathcal H})$ of the pair $({\mathcal G}, {\mathcal R}({\mathcal G}))$ is a map $\pi:{\mathcal G}\to {\mathcal B}({\mathcal H})$ such that the collection $\{\pi(g):\;g\in {\mathcal G}\}$ satisfies all the relations of ${\mathcal R}({\mathcal G}).$  The smallest $C^{\ast}$-subalgebra of ${\mathcal B}({\mathcal H})$ containing $\{\pi(g):\;g\in {\mathcal G}\}$ is called a $C^{\ast}$-algebra that represents  $({\mathcal G}, {\mathcal R}({\mathcal G}));$ we denote this $C^{\ast}$-algebra by ${\mathcal A}_{\pi}.$  A representation $(\pi_{\mathcal U}, {\mathcal H}_{\mathcal U})$ of $({\mathcal G}, {\mathcal R}({\mathcal G}))$ is said to be {\textbf{universal}} and ${\mathcal A}_{\pi_{\mathcal U}}$ is called the \textbf{universal $C^{\ast}$-algebra} associated to $({\mathcal G}, {\mathcal R}({\mathcal G}))$ if for every representation $(\pi, {\mathcal H})$ of the pair $({\mathcal G}, {\mathcal R}({\mathcal G}))$ there is a $\ast$-homomorphism $\rho:{\mathcal A}_{\pi_{\mathcal U}}\to {\mathcal A}_{\pi}$ satisfying
$$\pi(g)=\;\rho\circ \pi_{\mathcal U}(g),\;\forall g\in {\mathcal G}.$$  The general theory found in Blackadar (\cite{black}) can be used to show that this universal $C^{\ast}$-algebra exists, and is unique (the bounded-norm condition is used in %allows for 
the existence proof).

\end{defn}
\begin{example}
Let ${\mathcal G}=\{u, u^{\ast}, v, v^{\ast}\}$ and fix $\lambda\in \mathbb T$ with $\lambda=e^{2\pi i\alpha},\;\alpha\in [0,1).$  Let ${\mathcal R({\mathcal G})}$ consist of the following three identities, where $I$ denotes the identity operator in $B(\H)$:
\begin{itemize}
\item [(1)] $uu^{\ast}=u^{\ast}u=I.$
\item [(2)] $vv^{\ast}=v^{\ast}v=I.$
\item [(3)] $uv=\lambda vu.$
\end{itemize}
We note that relations (1) and (2) together with the $C^{\ast}$-norm  condition force $\|u\|=\|v\|=1.$   Relation (3) implies the universal $C^{\ast}$-algebra involved is noncommutative, and our universal $C^{\ast}$-algebra in this case is the well-known noncommutative torus ${\mathcal A}_{\alpha}$.
\end{example}

\begin{example}
Fix $N>1$ and let ${\mathcal G}=\{s_0, s_0^{\ast},\cdots, s_{N-1}, s_{N-1}^{\ast}\}.$  Let ${\mathcal R({\mathcal G})}$ consist of the relations
\begin{itemize}
\item  [(1)] $s_i^{\ast}s_i=1,\;0\leq i\leq N-1.$
\item  [(2)] $s_i^{\ast}s_j=0,\;0\leq i\not = j\leq N-1.$
\item  [(3)] $s_1s_1^{\ast}+s_2s_2^{\ast}+\cdots +s_{N-1}s_{N-1}^{\ast}=I.$
\end{itemize}
Again the first collection of relations (1) implies that $\|s_i\|=1,\;0\leq i\leq N-1,$ and also imply that the $s_i$ will be isometries and the $s_i^{\ast}$ will be partial isometries, $0\leq i\leq N-1.$\footnote{Recall that an isometry in $B(\H)$ is an operator $T$ such that $T^*T = I$; a partial isometry $S$ satisfies $S = SS^*S$. A projection in  $B(\H)$ is an operator that is both self-adjoint and idempotent. }
The universal $C^{\ast}$-algebra constructed via these generators and relations was first discovered by J. Cuntz in the late 1970's.  Therefore it is called the Cuntz algebra and is commonly denoted by ${\mathcal O}_N.$

\end{example}

We now wish to examine several different families of representations of ${\mathcal O}_N$ that take on a related form on the Hilbert space $L^2(\mathbb T)$ where $\mathbb T$ is equipped with Haar measure.  These types of representations were first studied by  Bratteli and  Jorgensen in \cite{BJA} and \cite{BJ1}, who found that cetain of these representations could be formed from wavelet filter functions.  They also appear as representations coming from inflated fractal wavelet filters and the recently defined monic representations of  Dutkay and Jorgensen (\cite{dutkay-jorgensen-fractals} and \cite{dutkay-jorgensen-monic}, respectively).

We now discuss a common theme for all of the representations of ${\mathcal O}_N$ mentioned above, as was first done by Bratteli and Jorgensen in \cite{BJ1}.  Fix $N>1,$ and suppose a collection of $N$ essentially  bounded measurable functions $\{h_0,h_1,\cdots,h_{N-1}\} \subseteq L^\infty(\T)$ is given.  %For $i\in \Z_N$, we
We define   bounded operators $\{T_i\}_{i=0}^{N-1}$  on $L^2(\mathbb T)$ associated to the functions  $\{h_0,h_1,\cdots,h_{N-1}\}$ by 
\begin{equation}
\label{Cuntzlike}
(T_i\xi)(z)\;=\;h_i(z)\xi(z^N),
\end{equation}
and we ask the question:  when do the $\{T_i\}_{i=0}^{N-1}$ give a representation of ${\mathcal O}_N$ on $L^2(\mathbb T)$?

We first compute that  for $i\in \Z_N = \{0, 1, \ldots, N-1\}$, the adjoint of each $T_i$ is given by
\begin{equation}
\label{Cuntzlikeadjoint}
(T_i^{\ast}\xi)(z)\;=\;\frac{1}{N}\sum_{\omega\in \mathbb T: \omega^N=z}\overline{h_i(\omega)}\xi(\omega).%\;i\in\mathbb Z_N
\end{equation}

If we denote the $N$ (measurable) branches of the $N^{\text{th}}$ root function by
$\tau_j:\mathbb T\to \mathbb T$, where
$$\tau_j(z=e^{2\pi it}) =e^{\frac{2\pi i(t+j)}{N}},\;t\in [0,1)\;\;\text{and}\;\; \;j\in\mathbb Z_N,$$
then we can rewrite our formula for $T_i^{\ast}$ as:
\begin{equation}
\label{Cuntzlikeadjoint2}
(T_i^{\ast}\xi)(z)\;=\;\frac{1}{N}\sum_{j=0}^{N-1}\overline{h_i(\tau_j(z))}\xi(\tau_j(z)). %\;i\in\mathbb Z_N
\end{equation}
(Note we have chosen specific branches for the $N^{\text{th}}$ root functions, but in our formula
for the adjoint $T_i^{\ast}$
we could have taken any measurable branches and obtained the same result.)

We now give necessary and sufficient conditions on the functions $\{h_0,h_1,\dots,h_{N-1}\},$ as stated in \cite{BJ1}, that the $\{T_i\}_{i=0}^{N-1}$  generate a representation of ${\mathcal O}_N.$
\begin{proposition}
\label{PropCuntzrep}
Fix $N>1,$ let $\{h_i\}_{i=0}^{N-1}\subset L^{\infty}(\mathbb T)$ and define $\{T_i\}_{i=0}^{N-1}$ as in Equation \eqref{Cuntzlike}.
Then the operators $\{T_i\}_{i=0}^{N-1}$ give a representation of the Cuntz algebra if and only if the map
\begin{equation}
\label{matrixfilt2}
z\;\mapsto\;\left(\frac{h_i\left(ze^{\frac{2\pi i j}{N}}\right)}{\sqrt{N}}\right)_{0\leq i,j \leq N-1}
\end{equation}
is a map from $\mathbb T$ into the unitary $N\times N$ matrices for almost all $z\in \mathbb T.$
\end{proposition}

\begin{proof}
See Section 1 of Bratteli and Jorgensen's seminal paper \cite{BJ1} for more details on this.
\end{proof}

The above proposition motivates the next definition:

\begin{defn}
\label{defCuntzlike}
Let $\{h_j\}_{j=0}^{N-1}$ be a subset of $L^{\infty}(\mathbb T).$
We say that this family is a {\bf Cuntz-like family} if the matrix of Equation \eqref{matrixfilt2} is unitary for almost all $z\in \mathbb T.$
\end{defn}

Bratteli and Jorgensen were the first to note, in \cite{BJA}, that certain wavelets on $L^2(\mathbb R),$ the so-called multiresolution analysis wavelets, could be used to construct %certain 
representations of the Cuntz algebra ${\mathcal O}_N,$ by examining the filter function families, and showing that they were ``Cuntz-like.''  Their representations used low- and high-pass filters associated to the wavelets to construct the related isometries as above. Filter functions on the circle $\mathbb T$ are used to define wavelets in the frequency domain (see \cite{Strichartz1} for an excellent exposition). We thus give our initial definitions of ``dilation-translation wavelet families" in the frequency domain rather than the time domain.  We note that we restrict ourselves to integer dilations on $L^2(\mathbb R);$  more general dilation matrices giving rise to unitary dilations on $L^2(\mathbb R^d)$ are described in the Strichartz article  \cite{Strichartz1}.

Fix an integer $N>1.$  Define the operator $D$ of dilation by $N$ on $L^2(\mathbb R)$ by:
\[
D(f)(t)=\sqrt{N}f(Nt)\;\;\text{for}\;\; f\in L^2(\mathbb R). 
\]
and define the translation operator $T$ on $L^2(\mathbb R)$ by
$$T(f)(t)=f\left(t-v\right)\;\;\text{for}\;\; f\in L^2(\mathbb R),\;\text{and let}\;T_v=[T]^v,\;v\in \mathbb Z.$$
 Let ${\mathcal F}$ denote the Fourier transform on $L^2\left(\mathbb R\right)$.
Set
$${\widehat{D}}={\mathcal F}D{\mathcal F}^*\;\;\text{and}\;\;\widehat{T}={\mathcal F}T{\mathcal F}^*.$$
Then
$$\widehat{D}(f)(x)=\frac{1}{\sqrt{N}}\,f\big(\frac{x}{N}\big)\;\; \text{and}\;\; \widehat{T}\left(f\right)\left(x\right)=e^{-2\pi ix}f\left(x\right)\;\;\text{for}\;\;f\in L^2\left(\mathbb R\right).$$

\begin{defn}
A {\bf {wavelet family in the frequency domain}} for dilation by $N>1$ is a subset $\{\Phi\}\cup\{\Psi_1,\cdots,\Psi_m\}\subseteq L^2\left(\mathbb R\right)$ such that
\begin{equation}
\label{eq:parental-wavelet-family}
\{\widehat{T}_v\left(\Phi\right):\;  v\in\mathbb Z\}\cup\{{\widehat{D}}^j \widehat{T}_v\left(\Psi_i\right):\;1\leq i \leq m,\;j\in\mathbb N,\;v\in\mathbb Z\}
\end{equation}
is an orthonormal basis for $L^2\left(\mathbb R\right)$. If $m = N-1$ and the set \eqref{eq:parental-wavelet-family}  %$\{\widehat{T}_v\left(\Phi\right): \;v\in\mathbb Z\}\cup\{{\widehat{D}}^j {\widehat {T}}_v\left(\Psi_i\right):\;1\leq i \leq N-1,\;j\in\mathbb Z,\;v\in\mathbb Z\}$
is an orthonormal basis for $L^2\left(\mathbb R\right)$, the family $\{\Phi\}\cup\{\Psi_1,\cdots,\Psi_{N-1}\}$ is called an {\bf orthonormal wavelet family} for dilation by $N$.
\end{defn}

In other words,  wavelet families are finite subsets of the unit ball of a Hilbert space $L^2(\mathbb R)$ that, when acted on by specific operators (in this case unitary operators corresponding to dilation and translation), give rise to a basis for the Hilbert space.

A fundamental algorithm for constructing  wavelet families is the concept of multiresolution analysis (MRA) developed by  Mallat and  Meyer in \cite{mallat}, and key tools for constructing the MRA's are filter functions for dilation by $N$.
\begin{defn}
\label{lowpassfilt}
Let $N$ be a positive integer greater than $1$.  A {\bf low-pass filter} $m_0$ for dilation by $N$ is a function  $m_0:\;\mathbb T\;\rightarrow\;\mathbb C$ which satisfies the following conditions:
\renewcommand{\labelenumi}{(\roman{enumi})}
\begin{enumerate}
\item $m_0\left(1\right)\;=\;\sqrt{N}$ (``low-pass condition'')
\item $\sum_{\ell=0}^{N-1}|m_0\left(ze^{\frac{2\pi i \ell}{N}}\right)|^2=N$ a.e.;
\item $m_0$ is H\"{o}lder continuous at $1;$
\item [ (iv) (Cohen's condition)] $m_0$ is  non-zero in a sufficiently large neighborhood of $1$ (%``Cohen's condition'' -- 
e.g. it is sufficient that $m_0$ be nonzero on the image of $[-\frac{1}{2N},\frac{1}{2N}]$ under the exponential map from $\mathbb R$ to $\mathbb T$).
\end{enumerate}
\end{defn}
Sometimes in the above definition, condition (iv) Cohen's condition is dropped and thus {\bf frame wavelets} are produced instead of orthonormal wavelets; these situations can be studied further in Bratteli and Jorgensen's book  \cite{BJ}.

Given a low-pass filter $m_0$ for dilation by $N$, we can naturally view $m_0$ as a $\mathbb Z$-periodic function on $\mathbb R$ by identifying $\mathbb T$ with $[0,1)$ and extending $\mathbb Z$-periodically.  Then there is a canonical way to construct a ``scaling function'' associated to the filter $m_0$. We set
$$\Phi\left(x\right)\;=\prod_{i=1}^{\infty}\left[\frac{m_0\left(N^{-i}\left(x\right)\right)}{\sqrt{N}}\right].$$

Then the infinite product defining $\Phi$ converges a.e. and gives an element of $L^2\left(\mathbb R\right)$. We call $\Phi$ a {\bf  scaling function in the frequency domain} for dilation by $N$. (The function ${\mathcal F}^{-1}(\Phi)=\phi$ is the scaling function in the sense of the original definition.)

Given a low-pass filter $m_0$ and the associated scaling function $\Phi$ for dilation by $N$, then if we have $N-1$ other functions defined on $\mathbb T$ which satisfy appropriate conditions described in the definition that follows,  we can construct the additional members of a wavelet family for dilation by $N$. %,  called the ``wavelets."
\begin{defn}
Let $N$ be a positive integer greater than $1$, and let  $m_0$  be a  low-pass filter for dilation by $N$ satisfying all the conditions of Definition \ref{lowpassfilt}. A set of essentially bounded  measurable $\mathbb Z$-periodic functions $m_1,m_2,\cdots,m_{N-1}$ defined on $\mathbb R$ are called {\bf high-pass filters} associated to $m_0$,
if
$$\sum_{\ell=0}^{N-1}{\overline{m_i\left(ze^{\frac{2\pi i \ell}{N}}\right)}}m_j\left(ze^{\frac{2\pi i \ell}{N}}\right)\;=\;\delta_{i,j}N\;\;\text{for}\;\; 0\leq\;i,j\;\leq N-1.$$
\end{defn}

Given a low-pass filter $m_0,$ it is always possible to find {\it measurable} functions $m_1,m_2,\cdots,m_{N-1}$ that serve as high-pass filters to $m_0.$
The functions $m_1,\; m_2,\cdots, m_{N-1}$ can then be seen as $\mathbb Z$-periodic functions on $\mathbb R$ as well.
The connection between filter functions and wavelet families was provided by  Mallat and  Meyer for $N=2$ in \cite{mallat} and then extended to more general dilation matrices.  We consider only integer dilations $N>1,$ and rely on the exposition of both  Strichartz \cite{Strichartz1} and Bratteli and Jorgensen (\cite{BJ}) in the material that follows below:
\begin{thm} \textup{(}
\cite{mallat}, \cite{Strichartz1} Section 1.5, \cite{BJ}\textup{)}
Let $N$ be a positive integer greater than $1$, let
$\left(m_0,m_1,\cdots, m_{N-1}\right)$ be a classical system of low and associated high-pass filters for dilation by $N$, where $m_0$ satisfies all the conditions of Definition \ref{lowpassfilt},
 and let
$\Phi$ be the scaling function in the frequency domain constructed from $m_0$ as above.  Then
\begin{equation}
\label{eq:filters&wavelets}
\{\Phi\}\;\bigcup\;\{ \Psi_1=\widehat{D}\left(m_1\Phi\right),\;\Psi_2=\widehat{D}\left(m_2\Phi\right),\;\cdots,\;\Psi_{N-1}=\widehat{D}\left(m_{N-1}\Phi\right)\}
\end{equation} 
is an orthonormal wavelet family in the frequency domain for dilation by $N$. The wavelets $\{ \Psi_1,\;\Psi_2,\;\cdots,\;\Psi_{N-1}\}$ are called the ``wavelets" in the frequency domain for dilation by $N$. If Cohen's condition is satisfied, the family \eqref{eq:filters&wavelets} is an orthonormal wavelet family. (Again, the functions $\{\psi_1={\mathcal F}^{-1}(\Psi_1), \psi_2={\mathcal F}^{-1}(\Psi_2),\cdots, \psi_{N-1}={\mathcal F}^{-1}(\Psi_{N-1})\}$ form the ``wavelets" in the original sense of the definition.)
\end{thm}
\begin{rmk}
It follows that filter systems are very important in the construction of wavelets arising from a multiresolution analysis.
In their proof of the result above, Bratteli and Jorgensen used a representation of the Cuntz algebra ${\mathcal O}_N$ arising from the filter system.
\end{rmk}

It is then clear the filter conditions expressed as above can just be formulated as stating that the functions $\{m_0,m_1,\cdots, m_{N-1}\}$ can be used to construct the following 
 function mapping $z\in\mathbb R/\mathbb Z\cong\mathbb T$ into the $N\times N$ unitary matrices over $\mathbb C$, given by the formula
\begin{equation}
\label{matrixfilt}
z\;\mapsto\;\left(\frac{m_j\left(ze^{\frac{2\pi i \ell}{N}}\right)}{\sqrt{N}}\right)_{0\leq j,\ell \leq N-1},
\end{equation}
and therefore give a Cuntz-like family in the sense of Definition \ref{defCuntzlike}.
\newline\newline
As noted earlier, Bratteli and Jorgensen proved that the operators $\{ S_i\}_{i=0}^{N-1}$ defined on $L^2(\mathbb T)$
by
\begin{equation}
(S_i\xi)(z)\;=\;m_i(z)\xi(z^N),
\end{equation}
for $\xi\in L^2(\mathbb T),\; z \in\mathbb T\;\;\text{and}\;\; i = 0, 1,\cdots , N-1,$ satisfy the relations
\begin{equation}
S_j^{\ast} S_i = \delta_{i,j}I,
\end{equation}
\begin{equation}
\sum_{i=0}^{N-1}S_iS_i^{\ast}\;=\;I,
\end{equation} which we saw in Proposition \ref{PropCuntzrep}; and thus we obtain exactly the Cuntz relations for the Cuntz algebra ${\mathcal O}_N.$

This gives the {\bf Bratteli-Jorgensen mapping} from  a wavelet family $\{\Phi\}\,\bigcup\,\{\Psi_1,\cdots,\Psi_{N-1}\}$ in $L^2(\mathbb R)$ arising from a multiresolution analysis into a representations of ${\mathcal O}_N.$

We now recall the inflated fractal wavelets of  Dutkay and  Jorgensen \cite{dutkay-jorgensen-fractals}, which also have a multiresolution analysis structure, and therefore also have related generalized filter functions that will satisfy Definition \ref{defCuntzlike} and a weakened low-pass condition. Thus, these filter functions will also give rise to representations of ${\mathcal O}_N$ on $L^2(\mathbb T)$.  We review here only the case where the fractals embed inside $[0,1],$ although the work in \cite{dutkay-jorgensen-fractals} generalizes to fractals sitting inside $[0,1]^d$ constructed from affine iterated function systems.  We note that a fine survey of the relationship between quadrature mirror filters of all types and representations of ${\mathcal O}_N$ can be found in the recent paper \cite{DPS}.

Fix an integer $N>1$.  Recall that  $\Z_N= \{0, 1, \cdots, N-1\}$; let $B\subset \Z_N$ be a proper  subset of $\Z_N.$ Recall from \cite{hutch} that there is a unique fractal set ${\bf F}\subset [0,1]$ satisfying
$${\bf F}=\bigsqcup_{i\in B}(\frac{1}{N}[{\bf F}+i]).$$
The Hausdorff dimension of ${\bf F}$ is known to be  $\log_N(|B|)$ (\cite{hutch}, Theorem 1 of Section 5.3).

\begin{defn}(\cite{dutkay-jorgensen-fractals})  Let $N,\; B\subset \mathbb N,$ and ${\bf F}$ be as described above.
We define the {\bf inflated fractal set} ${\mathcal R}$ associated to ${\bf F}$ by:
 $$\mathcal R=\bigcup_{j\in\mathbb Z}\,\bigcup_{v\in\mathbb Z}N^{-j}({\bf F}+v).$$  The Hausdorff measure $
 \mu$ of dimension $\log_N(|B|)$, restricted to $\mathcal R\subset \mathbb R,$ gives a Borel measure on ${\mathcal R}$,  but it is not a Radon measure on  ${\mathcal R},$ because bounded measurable subsets of ${\mathcal R}$ need not have finite $\mu$-measure.  A dilation operator $D$ and translation operators $\{T_v: v\in \mathbb Z\}$ on $L^2(\mathcal R,\mu)$ are defined as follows:  for $f\;\in\;L^2(\mathcal R,\mu)$,
$$D(f)(x)=\sqrt{|B|}f(Nx),\;$$
$$T_{v}(f)(x)=f(x-v).$$
\end{defn}

There is a natural multiresolution analysis (MRA) structure on $L^2({\mathcal R},\mu)$, which can be described as follows.  We define a scaling function or ``father wavelet" $\phi$ by  $\phi=\chi_{\bf F}$.   Translates of $\phi$ are orthonormal, and we define the core subspace $V_0$ of the MRA to be the closure of their span,
\[V_0 = \overline{\text{span}} \{ T_v(\phi): v \in \Z\}.\]
 For $j\in \Z$, set $V_j=D^j(V_0)$. It was shown in Proposition 2.8 of \cite{dutkay-jorgensen-fractals} (using slightly different notation) that $\bigcup_{j\in \mathbb Z}V_j$ is dense in $L^2({\mathcal R},\mu)$ and
$\bigcap_{j\in\mathbb Z} V_j=\{0\}.$  The inclusion $V_j\subset V_{j+1}$ follows from the fact
that
\begin{equation}
\label{refine}
\phi = \frac1{\sqrt{|B|}}\sum_{i\in B} DT_{i}(\phi).
\end{equation}
We note that the refinement equation (\ref{refine}) above gives a {\it weakened} low-pass filter for dilation by $N$,
defined by $h_0(z)=\sum_{i\in B}\frac1{\sqrt K}z^{i}$  for $z\in\mathbb T$.  It is weakened in that  conditions $(i)$ and $(iv)$  of Definition \ref{lowpassfilt} will not be satisfied in general, but it will satisfy
$$\sum_{\{w:\;w^N=z\}}|h_0(w)|^2=N\;\;\text{for $z\in \mathbb T$,}$$
and $h_0(z)$ will be non-zero in a neighborhood of $z=1.$
Using linear algebra, it is then possible to find $N-1$ corresponding ``high-pass" filters $\{h_1,h_2,\cdots, h_{N-1}\}$ defined as Laurent polynomials in $z$ (see Theorem 3.4 of \cite{DMP} for details) such that
the condition of Definition~\ref{defCuntzlike} is satisfied for the family $\{h_0,h_1, \cdots, h_{N-1}\},$ and one thus obtains a representation of ${\mathcal O}_N$ to go along with the wavelet family.
Moreover, the {high-pass filters} $\{h_1, \cdots, h_{N-1}\}$ are constructed in such a way to allow one to construct a subset $\{\psi_1, \psi_2, \cdots, \psi_{N-1}\}$ of $W_0=V_1\ominus V_0$ that serves as the generalized wavelet family for $L^2({\mathcal R}, \nu)$ in the sense that
$$\{D^jT_v(\psi_i):\; 1\leq i\leq N-1\;\;\text{and}\;\; j,v\in \mathbb Z\}$$
form an orthonormal basis for $L^2({\mathcal R}, \nu).$  See \cite{dutkay-jorgensen-fractals} and \cite{DMP} for further details on this construction.

Finally we wish to briefly discuss on the relationship of the above representations of ${\mathcal O}_N$ coming from Cuntz-like families of functions to the {\it monic} representations of ${\mathcal O}_N$ defined by Dutkay and Jorgensen in \cite{dutkay-jorgensen-monic}.

 Fix an integer $N>1.$ Let $K_N$ denote the infinite product space $\prod_{j=1}^{\infty}\mathbb Z_N,$ which has the topological structure of the Cantor set.
Denote by $\sigma$ the one-sided shift on $K_N:$
\begin{equation}\label{eq:sigma_KN}
\sigma\left((i_j)_{j=1}^{\infty}\right )= (i_{j+1} )_{j=1}^{\infty}
\end{equation}
and  let $\sigma_k,\;k\in \mathbb Z_N$ denote the inverse branches to $\sigma:$
\begin{equation}\label{eq:sigma_j_KN}
\sigma_k\left( (i_j)_{j=1}^{\infty}\right )\;=\; (ki_1i_2\cdots i_j\cdots )%_{j=1}^{\infty}.
\end{equation}

\begin{defn}
\label{defdjmonic}
(\cite{dutkay-jorgensen-monic})
A {\bf monic system} is a pair $(\mu, \{f_i\}_{i\in \mathbb Z_N}),$ where $\mu$ is a finite Borel measure on $K_N$ and $\{f_i\}_{i\in \mathbb Z_N}$ are  functions
on $K_N$ such that  for $j \in \mathbb Z_N$, $\mu\circ (\sigma_j)^{-1}<<\mu$ and
\begin{equation}
\label{monicsystem}
\frac{d \mu\circ (\sigma_j)^{-1}}{d \mu}\;=\;|f_j|^2,%\;j\in \mathbb Z_N,
\end{equation}
and the functions $\{f_j\}$ have the property that  
$$|f_j(x)|\;\not= \; 0,\;\mu-\text{a.e.}\;x\in \sigma_j(K_N).$$
A monic system is called  nonnegative if $f_j\geq 0$ for all $j\in \mathbb Z_N.$
\end{defn}

Given a monic system $(\mu, \{f_i\}_{i\in \mathbb Z_N}),$ in \cite{dutkay-jorgensen-monic} Dutkay and Jorgensen associated to it a representation of the Cuntz algebra ${\mathcal O}_N$ on $L^2(K_N,\mu)$ defined by:
$$S_j(\xi)(x)\;=\;f_j(x)\cdot \xi\circ \sigma(x)\;\;\text{for $\xi\in L^2(K_N,\mu)$ and $j\in\mathbb Z_N$},$$
and they proved that this representation is what they termed a {\it monic representation} (c.f. Theorem 2.7 of \cite{dutkay-jorgensen-monic} for details).

Recall we have a map $\iota: K_N\to \mathbb T$ defined by
$$\iota ((i_j)_{j=1}^{\infty})\;=\;e^{2\pi i\sum_{j=1}^{\infty}\frac{i_j}{N^j}}.$$  We also have an inverse map
$\theta:\mathbb T\to K_N$ where $\theta(e^{2\pi it})= (i_j)_{j=1}^{\infty}$ for $t=\sum_{j=1}^{\infty}\frac{i_j}{N^j}.$  We recall for the rational numbers there are more than one $N$-adic expansion, but such anomalies form a set of measure $0$ in $\mathbb T.$  With respect to this correspondence, the map
$\sigma$ looks like $\tau,$ where $\tau(z)=z^N,$ and the maps $\sigma_j$ correspond to the maps $\tau_j(e^{2\pi i t})=e^{2\pi i \frac{t+j}{N}},$ i.e.
\[
\iota\circ \sigma\;=\;\tau\circ \iota\;\;\text{and}\;\;\tau_j\circ \iota\;=\;\iota\circ\sigma_j\;\;\text{for}\;\;0\leq j\leq\;N-1.
\]
So, if the measure $\mu$ on $K_N$ is equal to Haar measure $\nu$ on $K_N$ (thought of as an infinite product of the cyclic groups $\Z_N$), a monic system of functions $\{f_i\}_{i\in \mathbb Z_N}$ on $(K_N,\nu)$ gives a collection of functions $\{h_i=f_i\circ \theta \}_{i\in \mathbb Z_N}$ on $\mathbb T.$

The most relevant aspect of Dutkay and Jorgensen's work on monic representations to this paper is that, using the fact that $(K_N,\nu)$ can be measure-theoretically identified with $(\mathbb T, \nu_{\mathbb T}),$ where $\nu_{\mathbb T} $ is Haar measure on the circle group $\mathbb T,$ by using the maps $\theta$ and $\iota$ defined above, it is possible to identify a  system of essentially bounded functions $\{h_j=f_j\circ \theta\}_{j=0}^{N-1}$  on  $\mathbb T,$ and one can check that these functions will satisfy the condition of Definition \ref{defCuntzlike}. The key relevant point in the proof of this is that by Theorem 2.9 of \cite{dutkay-jorgensen-monic}, the support of each $f_j$ is precisely $\sigma_j(K_N),$ and $|f_j((i_j)_{j=1}^{\infty})|^2=N$ on its support, so that the support of each $h_j$ is precisely $\tau_j(\mathbb T),$ with $|h_j(z)|=\sqrt{N}$ for $z\in \tau_j(\mathbb T)$ and $0$ otherwise. It therefore follows that monic systems of functions on $(K_N,\nu)$ moved over to $\mathbb T$ via the map $\theta$ all give rise to Cuntz-like systems of functions on $\mathbb T.$  However, these monic systems will only give rise to filter functions (and hence to classical wavelets) in isolated conditions (e.g., for $N=2$ it is possible to obtain the Shannon wavelet via a monic system of two functions that is equivalent to the filter functions $m_0(z)=\sqrt{2}\chi_{E_0}$ and $m_1(z)=\sqrt{2}\chi_{E_1},$ where $E_0$ is the image of $[0,\frac{1}{4})\cup [\frac{3}{4},1]$ under the exponential map from $[0,1]$ to $\mathbb T,$ and $E_1$ is the image in $\mathbb T$ of $[\frac{1}{4},\frac{3}{4})$).

Further analysis of monic representations can be found in \cite{dutkay-jorgensen-monic}.  We mention them here because they are the { closest} analog in the Cuntz $C^{\ast}$-algebra case to the sorts of representations of the higher-rank graph algebras that are used to construct wavelets in \cite{FGKP}.

\section{Marcolli-Paolucci wavelets}

 In the 2011 article \cite{MP}, Marcolli and Paolucci constructed representations of (finite) Cuntz-Krieger $C^{\ast}$-algebras on $L^2$-spaces for certain fractals, and then in certain cases went on to define wavelets generalizing the wavelets of A. Jonsson \cite{jonsson}.  We recall their basic constructions.

\begin{defn} Fix an integer $N>1.$ Let $A= \left(  A_{i,j}\right)_{i,j \in \Z_N}$ be a $N\times N$ matrix  whose entries $ A_{i,j}$ take on only values in $\{0,1\}$.  The {\bf Cuntz-Krieger $C^{\ast}$-algebra ${\mathcal O}_A$} is the universal $C^{\ast}$-algebra generated by partial isometries %\footnote{Recall that a partial isometry is an operator $T \in B(\H)$ such that $TT^*T = T.$} 
$\{T_i\}_{i\in {\Z}_N}$ satisfying
\begin{equation}
\label{CK1}
T_i^{\ast}T_i\;=\;\sum_{j=0}^{N-1}A_{i,j}T_jT_j^{\ast},
\end{equation}
\begin{equation}
\label{CK2}
T_i^{\ast}T_j\;=\;0\;\;\text{for $i\not=j$},
\end{equation}
and
\begin{equation}
\label{CK3}
\sum_{i=0}^{N-1}T_iT_i^{\ast}\;=\;\text{I}.
\end{equation}
\end{defn}

We note that these Cuntz-Krieger $C^{\ast}$-algebras ${\mathcal O}_A$ are  examples of $C^{\ast}$-algebras associated to certain special finite directed graphs, namely, those directed graphs admitting at most one edge with source $v$ and range $w$ for any pair of vertices $(v,w)$.  Indeed (cf.~\cite{Raeburn}, Remark 2.8) one can  show that the directed  graph in this case would have $N$ vertices in a set $E^0_A,$ labeled $E^0_A=\{v_0,v_1,\cdots, v_{N-1}\},$ with edge set $E^1_A=\{e_{(i,j)}\in \mathbb Z_N^2:\; A_{i,j}=1\}$; there is a (directed) edge $e_{(i,j)}$ beginning at $v_j$ and ending at $v_i$ iff $A_{i,j}=1.$
 The matrix $A$ then becomes the vertex matrix of the associated directed graph.
In the case where $A$ is the matrix that has $1$ in every entry, the $C^{\ast}$-algebra ${\mathcal O}_A$ is exactly the Cuntz algebra ${\mathcal O}_N.$

As had been done previously by K. Kawamura \cite{kawamura},  Marcolli and Paolucci constructed representations of ${\mathcal O}_A$ by employing the method of ``semibranching function systems".  We note for completeness that the semibranching function systems of Kawamura (\cite{kawamura}) were for the most part defined on finite Euclidean spaces, e.g. the unit interval $[0,1],$ whereas the semibranching function systems used by Marcolli and Paolucci (\cite{MP}) were mainly defined on Cantor sets.

\begin{defn} (c.f. \cite{kawamura}, \cite{MP} Definition 2.1, \cite{bezuglyi-jorgensen} Theorem 2.22) % note that Kawamura doesn't distinguish this defn with a number.

\label{semibranchingdef}
Let $(X,\mu)$ be a measure space and let $\{D_i\}_{i\in \mathbb Z_N}$ be a collection of $\mu$-measurable sets and $\{\sigma_i:D_i\to X\}_{i\in\mathbb Z_N}$ a collection of $\mu$-measurable maps. Let $A$ be an $N\times N$ $\{0,1\}$-matrix. The family of maps  $\{\sigma_i\}_{i\in\mathbb Z_N}$ is called a \emph{semibranching function system}
on $(X,\mu)$ with coding map $\sigma:X\to X$ if the following conditions hold:
\begin{enumerate}
  \item For $i\in \Z_N$, set $R_i=\sigma_i(D_i)$. Then we have 
  \[
  \mu(X\backslash \cup_{i\in\mathbb Z_N}R_i)=0\;\;\;\text{and}\;\;\; \mu(R_i\cap R_j)=0\;\;\text{for}\;\;i\not=j.
  \]
  \item  For $i\in \Z_N$, we have $\mu\circ \sigma_i<<\mu$ and
 \begin{equation}
 \label{semibranching}
 \frac{d(\mu\circ \sigma_i)}{d\mu}\;>\;0,\;\mu\text{-a.e. on}\;\;D_i.
 \end{equation}
  \item  For $i\in \Z_N$ and a.e. $ x\in D_i$, we have $$\sigma\circ \sigma_i(x)\;=\;x.$$
  \item (Cuntz-Krieger (C-K) condition:)  For $i,j\in \Z_N$, $\mu (D_i\Delta \cup_{i:A_{i,j}=1}R_j)\;=\;0.$
\end{enumerate}

\end{defn}

\begin{example} (\cite{kawamura}) Take $N>1,\;(X,\mu)=(\mathbb T,\nu)$ where $\nu$ is Haar measure on $\mathbb T,\;D_i=\mathbb T$ for $i\in \mathbb Z_N,$ and
$\sigma_j(z=e^{2\pi i t})=e^{\frac{2\pi i(x+j)}{N}}$ for $t\in [0,1);$ then $R_j=\{e^{2\pi i t}: t\in [\frac{j}{N},\frac{j+1}{N})\}.$ With the coding map  given by $\sigma(z)=z^N,$  we obtain a semibranching function system
{satisfying the  (C-K) condition}
for the $N\times N$ matrix consisting of all $1$'s.
\end{example}

\begin{example} (\cite{MP} Proposition 2.6)
\label{ex:mp-fractal}
  Take $N>1,$ and fix an $N\times N$ $\{0,1\}$-matrix $A.$ Let $\Lambda_A\subset \prod_{j=1}^{\infty}[\mathbb Z_N]_j$ be defined by

$$\Lambda_A\;=\;\{(i_1i_2\cdots i_j\cdots ):\; A_{i_j i_{j+1}}=1\;\;\text{for}\;\;j\in\mathbb N\}.$$
Marcolli and Paolucci have shown that, using the $N$-adic expansion map, $\Lambda_A$ can be embedded in $[0,1]$ as a fractal set and thus has a corresponding Hausdorff probability measure $\mu_A$ defined on its Borel subsets.
For each $i \in \Z_N$, let
$$D_i\;=\;\{(i_1i_2\cdots i_j\cdots ):\; A_{i,i_{1}}=1\}\;\subset \Lambda_A,$$
and define $\sigma_j$ for $j\in \mathbb Z_N$ by
$$\sigma_j:D_j\to \Lambda_A:\;\sigma_j\left( (i_1i_2\cdots i_k\cdots )\right )\;=\; (ji_1i_2\cdots i_k\cdots ).$$
Then
\[R_j := \sigma_j(D_j)=\{(j i_1 i_2\cdots i_k\cdots )\;: (i_1 i_2 \cdots i_k \cdots ) \in D_j\}, \]
and
denoting by $\sigma$ the one-sided shift on $\Lambda_A:$
$$\sigma\left( (i_1i_2\cdots i_k\cdots )\right )= (i_2i_3\cdots i_{k+1}\cdots )$$
we have that $\sigma\circ \sigma_j(x)=x$  for $x\in D_j$ and $j\in \mathbb Z_N.$
Marcolli and Paolucci show in Section 2.1 of \cite{MP} that this data  gives a semibranching function system satisfying the  (C-K) condition on  $(\Lambda_A,\mu_A).$  If $A$ is the matrix consisting entirely of $1$s, we obtain a monic system in the sense of \cite{dutkay-jorgensen-monic}.  Moreover, in this case, $D_i=K_N$ for all $i\in \mathbb Z_N$ and $R_i=Z(i)=\{(i_j)_{j=1}^{\infty}:\;i_1=i\}.$
\end{example}

Kawamura and then Marcolli and Paolucci observed the following relationship between semibranching function systems  satisfying the  (C-K) condition and representations of ${\mathcal O}_A:$

\begin{prop} (c.f. \cite{MP} Proposition 2.5)
\label{MPrepprop}
Fix a non-trivial $N\times N\;\{0,1\}$-matrix $A$ with $A_{i,i}=1$ for all $i\in \mathbb Z_N.$ Let $(X,\mu)$ be a measure space, and let $\{D_i\}_{i\in \mathbb Z_N},\;\{\sigma_i:D_i\to X\}_{i\in\mathbb Z_N}$ and $\{R_i=\sigma_i(D_i)\}_{i\in\mathbb Z_N}$ be a semibranching function system satisfying the  (C-K) condition on $(X,\mu)$ with coding map $\sigma:X\to X.$
For each $i\in \mathbb Z_N$ define $S_i:L^2(X,\mu)\to L^2(X,\mu)$ by
$$S_i(\xi)(x)\;=\;\chi_{R_i}(x)\Big(\frac{d\mu\circ \sigma_i}{d\mu}(\sigma(x))\Big)^{-\frac{1}{2}}\xi(\sigma(x))\;\;\text{for $\xi\in L^2(X,\mu)$ and $x\in X$}.$$
Then the family $\{S_i\}_{i\in \mathbb Z_N}$ of partial isometries satisfies the Cuntz-Krieger relations Equations (\ref{CK1}), (\ref{CK2}), and (\ref{CK3}), and therefore generates a representation of the Cuntz-Krieger algebra ${\mathcal O}_A.$
\end{prop}

We now discuss the construction of wavelets for Cuntz-Krieger $C^{\ast}$-algebras as developed by Marcolli and Paolucci.
In the setting of Example \ref{ex:mp-fractal}, suppose in addition that the matrix $A$ is irreducible, i.e.  for every pair $(i,j)\in {\mathbb Z}_N\times {\mathbb Z}_N$ there exists $n\in\mathbb N$ with $A^n_{i,j}\not=0.$

In this case, Marcolli and Paolucci proved that the Hausdorff measure $\mu_A$ on $\Lambda_A$ is exactly the probability measure associated to the normalized Perron-Frobenius eigenvector of $A$.  Namely, suppose $(p_0,p_1,\ldots, p_{N-1})$ is a  vector in $(0, \infty)^N$ %with $\ell^1$-norm 1
satisfying $\sum_{i=0}^{N-1}p_i=1,$  and such that
$$A(p_0,p_1,\ldots, p_{N-1})^T=\rho(A)(p_0,p_1,\cdots, p_{N-1})^T,$$
where $\rho(A)$ is the spectral radius of $A$.  (The existence of such a vector $(p_0, \ldots, p_{N-1})$, called the {\bf Perron-Frobenius eigenvector} of $A$, follows from the irreducibility of $A$.)  Then we have:
\begin{thm}
\label{thmMPmeasure}
(\cite{MP}, Theorem 2.17)
Let $N>1$ be fixed, and suppose  that $A$ is an irreducible $\{0,1\}$-matrix.
Let $\{\sigma_i:D_i\to R_i\}$  with $\sigma:\Lambda_A\to \Lambda_A$ be the semibranching  function system satisfying the  (C-K) condition associated to $(\Lambda_A, \mu_A)$ as in Example \ref{ex:mp-fractal}.  Then the Hausdorff measure $\mu_A$ on $\Lambda_A$ is exactly the probability measure associated to the  Perron-Frobenius eigenvector $(p_0, \ldots, p_{N-1})$ of $A$.  To be precise, for $i\in \Z_N$,  $\mu_A(R_i)\;=\;p_i$ and
$$\frac{d(\mu\circ \sigma_i)}{d\mu}\;=\;N^{-\delta_A},\;\text{a.e. on}\;\;D_i,$$
where $\delta_A$ is the Hausdorff dimension of $\Lambda_A,$ and
the spectral radius $\rho(A)$ of $A$ is equal to $N^{\delta_A}.$
\end{thm}

Given an irreducible $\{0,1\}$-matrix $A$ as in Theorem  \ref{thmMPmeasure},  Marcolli and Paolucci were able to construct families of ${\mathcal O}_A$-wavelets on $L^2(\Lambda_A,\mu_A)$ generalizing splines.  We describe their construction here (see also Section 3 of \cite{MP}).  For the purposes of this survey, we concentrate here on the wavelets whose scaling functions or ``father wavelets" are constant on the subsets $R_i$ of $\Lambda_A.$

We denote by ${\mathcal V}_0$ the (finite-dimensional) subspace of $L^2(\Lambda_A,\mu_A)$ given by
$${\mathcal V}_0\;=\;\overline{\text{span}}\{\chi_{R_i}:\;i\in\mathbb Z_N\}.$$
For each $k,\;0\leq k\leq N-1,$ let ${\mathcal D}_k=\;\{j:\;A_{kj}=1\},$ and let $d_k=|{\mathcal D}_k|.$ Enumerate the elements of ${\mathcal D}_k$ by setting  ${\mathcal D}_k=\{n_0<n_1<\cdots<{n_{d_k-1}}\}.$
For each $k\in \mathbb Z_N,$ define the following inner product on $\mathbb C^{d_k}:$
$$\langle (x_j), (y_j)\rangle_{PF}\;=\;\sum_{j=0}^{d_k-1}\overline{x_j}y_jp_{n_j},$$
where $(p_{n_j})$ are the appropriate coefficients of the Perron-Frobenius eigenvector of $A$.
We now define vectors $\{c^{j,k}:\;0\leq j \leq d_k-1,\;0\leq k\leq N-1\}$, where $c^{j,k} = (c^{j,k}_1, \ldots, c^{j,k}_{d_k -1})$, such that for each $k\in\mathbb Z_N,$
$\{c^{j,k}: 0\leq j\leq d_k-1\}$ is an orthonormal basis for $\mathbb C^{d_k-1}$ with respect to the inner product $\langle \cdot,\cdot \rangle_{PF},$ so that 
$$c_{\ell}^{0,k}=c_{\ell'}^{0,k}\;\;\text{for $0\leq \ell, \ell'\leq d_k-1$ and $k\in \mathbb Z_N$},$$
and for each fixed $k\in\mathbb Z_N,$
$$\text{span}\{c^{j,k}: 1\leq j\leq d_k-1\}\;=\;\{(1,1,\cdots,1)\}^{\perp}$$
with respect to the inner product $\langle \cdot,\cdot \rangle_{PF}$ defined above.

We now note that we can write each set $R_k$ as a disjoint union:
$$R_k=\bigcup_{j=0}^{d_k-1}R_{[kn_j]},$$
where
$$R_{[kn_j]}\;=\;\{(i_1i_2\cdots i_n\cdots )%_{n=1}^{\infty}
\in\Lambda_A:\;\;i_1=k\;\text{and}\;i_2=n_j\}.$$
Thus in terms of characteristic functions,
$$\chi_{R_k}\;=\;\sum_{j=0}^{d_k-1}\chi_{R_{[kn_j]}}\;\;\text{for $k\in\mathbb Z_N$}.$$
Now for each $k\in\mathbb Z_N$ we define functions $\{f^{j,k}\}_{j=0}^{d_k-1}$  on $\Lambda_A$ by
$$f^{j,k}(x)\;=\;\frac{1}{\sqrt{p_k}}\sum_{\ell=0}^{d_k-1}c_{\ell}^{j,k}\chi_{R_{[kn_{\ell}]}}(x).$$
We note that each function $f^{j,k}$ is $\mu_A$-measurable. %\st{since it is a finite linear combination of characteristic functions of measurable sets.}
Also, for each $k\in \mathbb Z_N,$
$$f^{0,k}=\frac{1}{\sqrt{p_{k}}}\sum_{\ell=0}^{d_k-1}c_{\ell}^{0,k}\chi_{R_{[kn_{\ell}]}} = \frac{c_1^{0,k}}{\sqrt{p_k}} \sum_{\ell=0}^{d_k-1} \chi_{R_{[kn_\ell]}} = \frac{c_1^{0,k}}{\sqrt{p_k}} \chi_{R_k}$$
is a scalar multiple of $\chi_{R_{k}}$, since the vector $c^{0,k}$ is a constant vector. 
 %since  $1$ is the common value of the $\{c_{\ell}^{0,k}\}_{\ell=0}^{d_k-1}.$ 
It follows that
$$\text{span}\{f^{0,k}\}_{k=0}^{N-1}\;=\;\text{span}\{\chi_{R_k}\}_{k=0}^{N-1}\;=\;{\mathcal V}_0.$$

We are now nearly ready to state our simplified version of the main theorem on wavelets in \cite{MP}.  First, a definition: Fix an integer $n > 1$.  We say that a word $w = w_1 w_2 \cdots w_n$ in $\prod_{k=1}^n \Z_N$ is \emph{admissible} for our $\{0,1\}$-matrix $A$ if, for all $1 \leq i \leq n-1$, we have $A_{w_i w_{i+1}} = 1$.  If $w$ is admissible, we write $S_w$ for the partial isometry in $B(L^2(\Lambda_A, \mu_A))$ given by
\[S_w = S_{w_1} S_{w_2} \cdots S_{w_n}.\]  We also remark that in order to be consistent with standard notation from multiresolution analysis theory and also with our notation for the higher-rank graph $C^{\ast}$-algebra wavelets, we have changed the notation for the orthogonal subspaces from the original notation used in \cite{MP}.

\begin{thm}
(\cite{MP}, Theorem 3.2)
\label{MPwaveletsthm}
Fix $N>1.$ Let $A$ be an $N \times N$, irreducible, $\{0,1\}$-matrix, let $(\Lambda_A,\mu_A)$ be the associated fractal space with Hausdorff measure, and let $\{\sigma_j:D_j\to R_j\}_{j\in \mathbb Z_N}$ and $\sigma$ be the associated semibranching function system satisfying the  (C-K) condition defined on $(\Lambda_A,\mu_A).$ Let $\{S_k\}_{k\in \Z_N}$ be a set of operators on $L^2(\Lambda_A, \mu_A)$ given by the formula in Proposition~\ref{MPrepprop}. Let $\{f^{j,k}:\;k\in\mathbb Z_N,\;0\leq j\leq d_k-1\}$ be the functions on $\Lambda_A$ defined in the above paragraphs. For $k\in \Z_N$, let
$$\phi_k\;=\;f^{0,k}.$$
Define
$${\mathcal W}_0\;=\;\overline{\text{span}}\{f^{j,k}: k\in \mathbb Z_N,\;1\leq j\leq d_{k-1}\};$$
$${\mathcal W}_n=\overline{\text{span}}\{S_w(f^{j,k}):\;k\in \mathbb Z_N,\;1\leq j\leq d_k-1\;\;\text{and}\;\;w\;\text{is an admissible word of length}\;n\}.$$
Then the subspaces ${\mathcal V}_0$ and $\{{\mathcal W}_n\}_{n=0}^{\infty}$ are mutually pairwise orthogonal in $L^2(\Lambda_A,\mu_A)$ and
$$L^2(\Lambda_A,\mu_A)=\; \text{span}\; \left({\mathcal V}_0\oplus \Big[\bigoplus_{n=0}^{\infty} \mathcal{W}_n\Big]\right).$$

\end{thm}

The $\phi_k$ are called the scaling functions (or ``father wavelets") and the $f^{j,k}$ are called the wavelets (or ``mother wavelets") for the system.

Since the proof of the above theorem can be read in \cite{MP}, we do not include it here.  However, as we did in the second paragraph of Section 4 of \cite{FGKP}, we do wish to remark upon the fact that the emphasis on the Perron-Frobenius measure in \cite{MP} does not appear to be crucial for the construction of the orthogonal subspaces.
To illustrate this further, we now construct wavelets for ${\mathcal O}_N$ corresponding to any Markov measure on $K_N,$ and here we will include the proof so as to illustrate our techniques.
Note also that by taking tensor products, the wavelets below will produce wavelets on $k$-graph algebras of tensor-product type, for example, 
in ${\mathcal O}_N \otimes {\mathcal O}_M,$ as studied in Example 3.8 of \cite{KangPask}.

 Recall from Section 2 that $K_N$ is the infinite product space $\prod_{j=1}^\infty \Z_N$ which can be realized as the Cantor set. Let $\sigma$ and $\{\sigma_j\}_{j\in \Z_N}$ be the one-sided shift on $K_N$ and its inverse branches given in \eqref{eq:sigma_KN} and \eqref{eq:sigma_j_KN}.   Following Example 3.11 of \cite{dutkay-jorgensen-monic}, fix $\{p_i\in (0,1): i\in\mathbb Z_N\}$, with $\sum_{i\in \mathbb Z_N}p_i=1,$ and define the Markov measure
$$\mu(Z(i_1i_2\cdots i_n))=\prod_{j=1}^np_{i_j},$$
where $i_j\in \Z_N$ for $1\le j \le n$, and 
$$Z(i_1i_2\cdots i_n)\;=\;\{(x_1x_2\cdots x_j\cdots )%_{j=1}^{\infty}
:\; x_{1}=i_1,\;x_2=i_2,\cdots,\;x_n=i_n\}.$$
As described at the end of Example \ref{ex:mp-fractal}, for the $N\times N$ matrix consisting of all $1$'s, the standard semibranching function system on $K_N$ satisfying the  (C-K) condition is  given by $\{\sigma_i: D_i\to R_i\}_{i\in Z_N},$ where $\sigma: K_N\to K_N$ satisfies $D_i=K_N$ for all  $i\in\mathbb Z_N$ and $R_i=Z(i).$
\begin{thm}
\label{MPwaveletsMarkov}
Fix $N>1,$ let $\{p_i\}_{i=0}^{N-1}$ be a collection of positive numbers with $\sum_{i=0}^{N-1}p_i=1,$ and let $\mu$ be the associated Markov Borel probability measure on $(K_N,\mu)$ defined as above.   For $i\in \Z_N$, let $\{\sigma_i:K_N\to R_i=Z(i)\}$ and $\sigma:K_N\to K_N$ be the associated semibranching function system satisfying the  (C-K) condition associated to the $N\times N$ matrix of all $1's$, and define $S_i \in B(L^2(K_N, \mu))$ by
\[S_i(f)(w) = \chi_{Z(i)}(w) p_i^{-1/2} f(\sigma(w)).\]
%\footnote{Elizabeth isn't sure that we know that the maps operators $S_i$ are isometries unless $\mu$ is the Perron-Frobenius measure,  but I don't think we need that in order to get the wavelet decomposition.} \]}
 Then as in Theorem \ref{MPwaveletsthm}, there are scaling functions $\{\phi_k\}_{k=0}^{N-1}\subset L^2(K_N,\mu)$ and ``wavelets"  $\{\psi_{j,k}:\;k\in\mathbb Z_N,\;1\leq j\leq N-1\}$ such that setting
$${\mathcal V}_0=\overline{\text{span}}\{\phi_k:\;k\in\mathbb Z_N\},$$
$${\mathcal W}_0\;=\;\overline{\text{span}}\{\psi_{j,k}:k\in \mathbb Z_N,\;1\leq j\leq N-1\}\;\;\text{and}$$
$${\mathcal W}_n=\overline{\text{span}}\{S_w(\psi_{j,k}):k\in \mathbb Z_N,\;1\leq j\leq N-1,\;w\;\text{a word of length}\;n\}\;\;\text{for}\;\; n\geq 1,$$
 we obtain
$$L^2(K_N,\mu)=\;{\text{span}} \; \left({\mathcal V}_0\oplus [\bigoplus_{n=0}^{\infty}{\mathcal W}_n]\right).$$
\end{thm}

\begin{proof}
Following the method of Theorem \ref{MPwaveletsthm}, we define an inner product on $\mathbb C^N$ by setting
\begin{equation}\langle (x_j)_j, (y_j)_j\rangle\;=\;\sum_{j\in \mathbb Z_N}\overline{x_j}\cdot y_j\cdot p_j.
\label{eq:4.9innerprod}
\end{equation}
For fixed $k\in \mathbb Z_N,$ we let $c^{0,k}$ be the vector in $\mathbb C^N$ defined by
$$c^{0,k}=(1,1,\cdots,1),$$
and let $\{c^{j,k}\}_{1\leq j \leq N-1}$, with $c^{j,k}=(c^{j,k}_{\ell}  )_{\ell \in \Z_N}$,  be any orthonormal basis for $\{(1,1,\cdots,1)\}^{\perp}$ with respect to the inner product \eqref{eq:4.9innerprod}.
For fixed $k\in \mathbb Z_N,$ define functions $\{f^{j,k}: 0\;\leq j\leq N-1\}$  on $K_N$ by:
$$f^{j,k}(x)\;=\;\frac{1}{\sqrt{p_k}}\sum_{\ell=0}^{N-1}c^{j,k}_{\ell}\chi_{Z(k\ell)}(x).$$
Note that $f^{0,k}$ is a normalized version of $\chi_{Z(k)}(x).$
We claim that setting
$$\phi_{k}(x)=f^{0,k}(x)\;\;\text{for $k\in \Z_N$},$$
and
$$\psi_{j,k}(x)=f^{j,k}(x)\;\;\text{for $1\leq k\leq N-1$ and $1\leq j\leq N-1$},$$
we will obtain a wavelet family for $L^2(K_N,\mu)$ where $\mu$
is the Markov measure determined by %the described weights.
\[ \mu(Z(i_1 i_2 \cdots i_n)) = \prod_{j=1}^n p_{i_j}.\]

We first note that if $i_1\not=i_2,$
the integral
$$\int_{K_N}\phi_{i_1}\overline{\phi_{i_2}}d\mu$$
is a scalar multiple of the integral
$$\int_{K_N}\chi_{Z(i_1)}(x)\chi_{Z(i_2)}(x)d\mu$$
and this latter integral is equal to zero because the functions in question have disjoint support.

We also remark that for $1\leq j\leq N-1$ and $k\in \mathbb Z_N,$
$$\sum_{\ell=0}^{N-1}c^{j,k}_{\ell}p_{\ell}=0.$$
Multiplying by $p_k$ we get:
$$\sum_{\ell=0}^{N-1}c^{j,k}_{\ell}p_{\ell}p_k=0,\;\;\text{so that}\;\;\int_{K_N}\Big[\sum_{\ell=0}^{N-1}c^{j,k}_{\ell}\chi_{Z(k\ell)}(x)\Big] d\mu=0.$$
We can write this as:
$$\int_{K_N}\Big[\sum_{\ell=0}^{N-1}c^{j,k}_{\ell}\chi_{Z(k\ell)}(x)\Big]\overline{\chi_{Z(k)}(x)}d\mu=0,$$
i.e.
$$\int_{K_N}\psi_{j,k}(x)\overline{\phi_k(x)}d\mu=0\;\;\text{for $1\leq j\leq N-1$}.$$

We now check and calculate: 
\[\begin{split}
S_i(\psi_{j,k})(x)\;&=\;S_i\Big(\frac{1}{\sqrt{p_k}}\sum_{\ell=0}^{N-1}c^{j,k}_{\ell}\chi_{Z(k\ell)}\Big)(x)=\;\frac{1}{\sqrt{p_k}}\sum_{\ell=0}^{N-1}c^{j,k}_{\ell}S_i(\chi_{Z(k\ell)})(x)\\
&=\;\frac{1}{\sqrt{p_k}}\sum_{\ell=0}^{N-1}c^{j,k}_{\ell}\chi_{Z(i)}(x)\frac{1}{\sqrt{p_i}}\chi_{Z(k\ell)}(\sigma(x))\\
&=\;\frac{1}{\sqrt{p_k}}\frac{1}{\sqrt{p_i}}\sum_{\ell=0}^{N-1}c^{j,k}_{\ell}\chi_{Z(ik\ell)}(x).
\end{split}
\]
Let ${\mathcal V}_0\;=\;\text{span}\{\phi_i:\;i\in\mathbb Z_N\}$ and let
$${\mathcal W}_0\;=\;\text{span}\{\psi_{j,k}:\;k\in\mathbb Z_N,\;1\leq j\leq N-1\}.$$
We have shown ${\mathcal V}_0\perp {\mathcal W}_0$.
We now define
$${\mathcal W}_1\;=\;\text{span}\{S_i(\psi_{j,k}): i,\;k\in \mathbb Z_N,\;1\leq j\;\leq N-1\}.$$
A straightforward calculation shows that 
$\{S_i(\psi_{j,k}): i,\;k\in \mathbb Z_N,\;1\leq j\;\leq N-1\}$ is an orthonormal set of functions.  

We prove now that $({\mathcal V}_0\oplus {\mathcal W}_0)\perp {\mathcal W}_1.$

Let us first fix pairs $(j,k)$ and $(j',k')$ with
$k,\;k'\in \mathbb Z_N$ and $1\leq j,\; j'\leq N-1.$
Fix $i\in \mathbb Z_N.$ Then 
\[\begin{split}
&\int_{K_N}S_i(\psi_{j,k})(x)\overline{\psi_{j',k'}(x)}d\mu\;=\;\frac{1}{\sqrt{p_i}}\int_{K_N}\Big[\frac{1}{\sqrt{p_k}}\sum_{\ell=0}^{N-1}c^{j,k}_{\ell}\chi_{Z(ik\ell)}(x)\Big]\overline{\psi_{j',k'}(x)}\, d\mu\\
&=\;\frac{1}{\sqrt{p_i}}\int_{K_N}\Big[\frac{1}{\sqrt{p_k}}\sum_{\ell=0}^{N-1}c^{j,k}_{\ell}\chi_{Z(ik\ell)}(x)\Big]\overline{\frac{1}{\sqrt{p_{k'}}}\sum_{\ell'=0}^{N-1}c^{j',k'}_{\ell'}\chi_{Z(k'\ell')}(x)}d\mu\\
&=\;\frac{1}{\sqrt{p_i}}\frac{1}{\sqrt{p_kp_{k'}}}\int_{K_N}\sum_{\ell=0}^{N-1}\sum_{\ell'=0}^{N-1}\delta_{i,k'}\delta_{k,\ell'}c^{j,k}_{\ell}\overline{c^{j',k'}_{\ell'}}\chi_{Z(ik\ell)}(x)d\mu\\
&=\;\frac{1}{\sqrt{p_i}}\frac{1}{\sqrt{p_kp_{k'}}}\int_{K_N}\sum_{\ell=0}^{N-1}\delta_{i,k'}
c^{j,k}_{\ell}\overline{c^{j',k'}_{k}}\chi_{Z(ik\ell)}(x)d\mu\;=\frac{1}{\sqrt{p_i}}p_ip_k\delta_{i,k'}\frac{1}{\sqrt{p_kp_{k'}}}\sum_{\ell=0}^{N-1}c^{j,k}_{\ell}\overline{c^{j',k'}_{k}}p_{\ell}\\
&=\frac{1}{\sqrt{p_i}}p_ip_k\delta_{i,k'}\overline{c^{j',k'}_{k}}\frac{1}{\sqrt{p_kp_{k'}}}\sum_{\ell=0}^{N-1}c^{j,k}_{\ell}p_{\ell}\;=\;0,
\end{split}\]
since $\sum_{\ell=0}^{N-1}c^{j,k}_{\ell}p_{\ell}=0$ for $1\leq j\leq N-1.$

It follows that
$${\mathcal W}_0\perp {\mathcal W}_1.$$
We now show that ${\mathcal V}_0\perp {\mathcal W}_1.$
Fix $i\in\mathbb Z_N$. %and $\phi_i.$  
Let $i',\,k\in\mathbb Z_N$ and fix $j\in\{1,2,\cdots, N-1\}.$  Then:
\begin{align*}\langle \phi_i, S_{i'}(\psi_{j,k}) \rangle\; & =\;\frac{1}{\sqrt{p_i}}\int_{K_N}\chi_{Z(i)}(x)\overline{\frac{1}{\sqrt{p_{i'}}}\frac{1}{\sqrt{p_k}}\sum_{\ell=0}^{N-1}c^{j,k}_{\ell}\chi_{Z(i'k\ell)}(x)}d\mu\\
&=\; \frac{1}{\sqrt{p_ip_{i'}}}\frac{1}{\sqrt{p_k}}\sum_{\ell=0}^{N-1}\int_{K_N} \chi_{Z(i)}(x) \overline{c^{j,k}_{\ell}}\chi_{Z(i'k\ell)}(x)d\mu\\
&=\;\frac{1}{\sqrt{p_ip_{i'} p_k}}\delta_{i,i'}\sum_{\ell=0}^{N-1}\int_{K_N}  \overline{c^{j,k}_{\ell}}\chi_{Z(i'k\ell)}(x)d\mu\\
& =\;\delta_{i,i'}\frac{1}{\sqrt{p_ip_{i'}p_k}}\sum_{\ell=0}^{N-1}\overline{c^{j,k}_{\ell}}p_{i'}p_kp_{\ell} \\
& =\;\delta_{i,i'}\frac{1}{\sqrt{p_ip_{i'}p_k}}p_{i'}p_k\Big[\sum_{\ell=0}^{N-1}\overline{c^{j,k}_{\ell}}p_{\ell}\Big].\end{align*}
In order for this value to have a chance of being nonzero we need $i=i'$.  But even if that happens we get:
$$\langle \phi_i, S_{i}(\psi_{j,k})\rangle=p_k\Big[\frac{1}{\sqrt{p_k}}\sum_{\ell=0}^{N-1}\overline{c^{j,k}_{\ell}}p_{\ell}\Big]=
\sqrt{p_k}\cdot \langle c^{j,k}, (1,1,\cdots, 1)\rangle_{\mathbb C^N},$$ 
which is equal to $0$ for $j\in \mathbb Z_N\backslash \{0\}.$  Thus $\mathcal{V}_0$ is orthogonal to $\mathcal{W}_1$.

We now prove by induction that if for every $n\geq 0$ we define
$${\mathcal W}_n\;=\;\text{span}\{S_w(\psi_{j,k}):\;w\;\text{ is a word of length}\;n,\;k\in\mathbb Z_N\;\;\text{and}\;\;1\leq j\leq N-1\},$$
then for all $n\geq 0,$
$${\mathcal W}_{n+1}\perp \Big[{\mathcal V}_0\oplus \bigoplus_{k=0}^n{\mathcal W}_k\Big].$$
We have proven this for $n=0$ directly.  We now assume it is true for $\ell=n$ and prove it is true for $\ell=n+1,$ i.e. let us prove that
${\mathcal W}_{n+2}$ is orthogonal to $ [{\mathcal V}_0\oplus \bigoplus_{k=0}^{n+1}{\mathcal W}_k].$
We first note that if $w$ is a word of length $n+2$ and $w'$ is a word of length $s$ where $1\leq  s\leq n+1,$ and if $k, k'\in \mathbb Z_N$ and $1\leq j, j'\leq N-1,$ then there are unique $i, i'\in \Z_N$ such that
$$\langle S_w(\psi_{j,k}), S_{w'}(\psi_{j',k'})\rangle_{L^2(K_N,\mu)}=\;\langle S_iS_{w_1}(\psi_{j,k}), S_{i'}S_{w_1'}(\psi_{j',k'})\rangle_{L^2(K_N,\mu)}.$$
where $w_1$ is a word of length $n+1$ and $w_1'$ is a word of length $s-1$ for $0\leq s-1\leq n.$
This then is equal to
$$\langle S_{w_1}(\psi_{j,k}), S_{i}^*S_{i'}S_{w_1'}(\psi_{j',k'})\rangle_{L^2(K_N,\mu)};$$
{since $S_i(\psi)(w) = \chi_{Z(i)}(x) p_i^{-1/2} \psi(\sigma(x))$, one can check that
\[S_i^*(\psi)(w) = p_i^{1/2} \psi(iw).\]}

%\st{By the Cuntz relations,}
 It follows that $S_{i}^*S_{i'} = \delta_{i, i'} I$.   If $i=i'$ so that $S_{i}^*S_{i'} = I, $ %is equal to the identity,
we obtain:
$$\langle S_w(\psi_{j,k}), S_{w'}(\psi_{j',k'})\rangle_{L^2(K_N,\mu)} = \langle S_{w_1}(\psi_{j,k}), S_{w_1'}(\psi_{j',k'})\rangle_{L^2(K_N,\mu)},$$
which is equal to $0$ by the induction hypothesis.

Thus, in either case,
$${\mathcal W}_{n+2}\perp \Big[\bigoplus_{k=1}^{n+1}{\mathcal W}_k\Big].$$
Now suppose $\psi_{j,k}\in{\mathcal W}_0$, $w$ is a word of length $n+2$, $k'\in \mathbb Z_N$ and $1\leq j'\leq N-1.$
Then,
$$\langle \psi_{j,k}, S_w(\psi_{j',k'})\rangle_{L^2(K_N,\mu)}\;=\;
%\langle f^{j,k}, S_w(\psi_{j',k'})\rangle_{L^2(K_N,\mu)}$$
\langle\frac{1}{\sqrt{p_k}}\sum_{\ell=0}^{N-1}c^{j,k}_{\ell}\chi_{Z(k\ell)}, S_w(\psi_{j',k'})\rangle_{L^2(K_N,\mu)}.$$
Write $w=i_1i_2\cdots i_{n+1}i_{n+2}$.  Then
$$S_w(\chi_{Z(k'\ell')})\;=\;S_{i_1}S_{i_2}\dots S_{i_{n+2}}(\chi_{Z(k'l')})\;=\;\frac{1}{\prod_{v=1}^{n+2}\sqrt{p_{i_v}}}\chi_{Z(i_1i_2\cdots i_{n+1}i_{n+2}k'\ell')},$$ so that
 $$\langle\frac{1}{\sqrt{p_k}}\sum_{\ell=0}^{N-1}c^{j,k}_{\ell}\chi_{Z(k\ell)}, S_w(\psi_{j',k'})\rangle_{L^2(K_N,\mu)}\; \qquad\qquad \qquad \qquad \qquad  \qquad$$
 $$ \quad \quad  =\;\frac{1}{\prod_{v=1}^{n+2}\sqrt{p_{i_v}}}\frac{1}{\sqrt{p_kp_{k'}}}\sum_{\ell=0}^{N-1}\sum_{\ell'=0}^{N-1} \langle c^{j,k}_{\ell}\chi_{Z(k\ell)},\;c^{j',k'}_{\ell'}\chi_{Z((i_1i_2\cdots i_{n+1}i_{n+2}k'\ell')}\rangle_{L^2(K_N,\mu)} \qquad \qquad$$
 $$\quad=\;\frac{1}{\prod_{v=1}^{n+2}\sqrt{p_{i_v}}}\frac{1}{\sqrt{p_kp_{k'}}}\sum_{\ell=0}^{N-1}\sum_{\ell'=0}^{N-1} \delta_{k,i_1}\delta_{\ell, i_2}\langle c^{j,k}_{\ell}\chi_{Z(k\ell)},\;c^{j',k'}_{\ell'}\chi_{Z((i_1i_2\cdots i_{n+1}i_{n+2}k'\ell')}\rangle_{L^2(K_N,\mu)}$$
 $$=\;\frac{1}{\prod_{v=1}^{n+2}\sqrt{p_{i_v}}}\delta_{k,i_1}\frac{1}{\sqrt{p_kp_{k'}}}\sum_{\ell'=0}^{N-1} \langle c^{j,k}_{i_2}\chi_{Z(ki_2)},\;c^{j',k'}_{\ell'}\chi_{Z(i_1i_2\cdots i_{n+1}i_{n+2}k'\ell')}\rangle_{L^2(K_N,\mu)}. \quad\quad \;\;$$
 %\st{If $k\not=i_1,$ we get the quantity equal to $0.$
%We now assume $k=i_1.$ }
 { This quantity will only be nonzero if $k = i_1$;} in this case we get:
$$\langle \psi_{j,k}, S_w(\psi_{j',k'})\rangle_{L^2(K_N,\mu)}\;=\;\frac{1}{\prod_{v=1}^{n+2}\sqrt{p_{i_v}}}\frac{1}{\sqrt{p_{i_1}p_{k'}}}\sum_{\ell'=0}^{N-1} \langle c^{j,i_1}_{i_2}\chi_{Z(i_1i_2)},\;c^{j',k'}_{\ell'}\chi_{Z((i_1i_2\cdots i_{n+1}i_{n+2}k'\ell')}\rangle_{L^2(K_N,\mu)}$$
$$\;=\;\frac{1}{\prod_{v=1}^{n+2}\sqrt{p_{i_v}}}\frac{1}{\sqrt{p_{i_1}p_{k'}}}\sum_{\ell'=0}^{N-1}\int_{K_N} c^{j,i_1}_{i_2}\overline{c^{j',k'}_{\ell'}}\chi_{Z(i_1i_2)}(x)\chi_{Z((i_1i_2\cdots i_{n+1}i_{n+2}k'\ell')}(x)d\mu \qquad \qquad $$
$$\;=\;\frac{1}{\prod_{v=1}^{n+2}\sqrt{p_{i_v}}}\sqrt{p_{k'}}\frac{1}{\sqrt{p_{i_1}}}\Big(\prod_{v=1}^{n+2}p_{i_v}\Big)c^{j,i_1}_{i_2}
\sum_{\ell'=0}^{N-1}\overline{c^{j',k'}_{\ell'}}p_{\ell'}$$
$$\;=\;\frac{\sqrt{p_{k'}}}{\sqrt{p_{i_1}}}c^{j,i_1}_{i_2}\Big(\prod_{v=1}^{n+2}\sqrt{p_{i_v}}\Big)\sum_{\ell'=0}^{N-1}1\cdot \overline{c^{j',k'}_{\ell'}}p_{\ell'}\;=\;0.$$
So in all cases, $\langle \psi_{j,k}, S_w(\psi_{j',k'})\rangle_{L^2(K_N,\mu)} = 0$, and we have ${\mathcal W}_{n+2}\perp {\mathcal W}_0.$

Finally, we want to show that ${\mathcal W}_{n+2}\perp {\mathcal V}_0.$  Let $\phi_k\in {\mathcal V}_0$ be fixed and let $S_w(\psi_{j',k'})\in {\mathcal W}_{n+2}$ for $w=i_1i_2\cdots i_{n+2}$ a word of length $n+2$ and
$k'\in\mathbb Z_N,\;j'\in\{1,2,\cdots, N-1\}.$
Then
$$\langle \phi_k, S_w(\psi_{j',k'})\rangle_{L^2(K_N,\mu)}\;=\;
\frac{1}{\sqrt{p_k}}\int_{K_N}\chi_{Z(k)}\overline{\Big(\prod_{v=1}^{n+2}\frac{1}{\sqrt{p_{i_v}}}\Big)\frac{1}{\sqrt{p_{k'}}}\sum_{\ell=0}^{N-1} c^{j',k'}_{\ell}\chi_{Z(i_1i_2\cdots i_{n+1}i_{n+2}k'\ell)}  }d\mu$$
$$=\;\frac{1}{\sqrt{p_kp_{k'}}}\Big(\prod_{v=1}^{n+2}\frac{1}{\sqrt{p_{i_v}}}\Big)\delta_{k,i_1}\int_{K_N}\overline{\sum_{\ell=0}^{N-1} c^{j',k'}_{\ell}\chi_{Z(i_1i_2\cdots i_{n+1}i_{n+2}k'\ell)}  }d\mu$$
$$\;=\;\frac{1}{\sqrt{p_kp_{k'}}}\Big(\prod_{v=1}^{n+2}\frac{1}{\sqrt{p_{i_v}}}\Big)\delta_{k,i_1}\Big[\prod_{v=1}^{n+2}p_{i_v}\Big]p_{k'}\sum_{\ell=0}^{N-1} \overline{c^{j',k'}_{\ell}}p_{\ell}$$
$$\;=\;\delta_{k,i_1}\frac{1}{\sqrt{p_k}}\sqrt{p_{k'}}\Big[\prod_{v=1}^{n+2}\sqrt{p_{i_v}}\Big]\sum_{\ell=0}^{N-1}1\cdot  \overline{c^{j',k'}_{\ell}}p_{\ell}\;=\;0.$$
It follows that ${\mathcal W}_{n+2}\perp {\mathcal V}_0,$ and we have proved the desired result by induction.
\end{proof}

 \begin{rmk} Notice that the  proof of Theorem \ref{MPwaveletsMarkov} also extends to any other  measure with shift operators having constant Radon-Nykodym derivative on cylinder sets.
\end{rmk}

\section{$C^*$-algebras corresponding to directed graphs and higher-rank graphs}

\subsection{Directed graphs, higher-rank graphs and $C^*$-algebras}
 A \textbf{directed graph} $E$ consists of a countable collection of vertices $E^0$ and edges $E^1$ with range and source maps $r,s: E^1\rightarrow E^0$. We view an edge $e$ as being directed from its source $s(e)$ to its range $r(e)$. A \textbf{path} is a string of edges $e_1e_2\dots e_n$ where $s(e_i)=r(e_{i+1})$ for $i=1,2,\dots, n-1$. The \textbf{length} of a path is the number of edges in the string.
As mentioned in the introduction, the graph $C^*$-algebra $C^*(E)$ is the universal $C^*$-algebra generated by a set of projections $\{p_v: v\in E^0\}$ and a set of partial isometries $\{s_e: e\in E^1\}$ that satisfy the Cuntz-Krieger relations. (These are relations (CK1)--(CK4) in \eqref{eq:CK1--4} below).

Higher-rank graphs, also called $k$-graphs, are  higher-dimensional analogues of directed graphs.
By definition, a higher-rank graph is a small category $\Lambda$ with a functor $d$ from the set $\Lambda$ of morphisms to $\N^k$ satisfying the \textbf{factorization property} : if $d(\lambda)=m+n$, then there exist unique $\alpha,\beta\in \Lambda$ such that $d(\alpha)=m$, $d(\beta)=n$, and $\lambda=\alpha\beta$.\footnote{We think of $\N^k$ as a category with one object, namely $0$, and with composition of morphisms given by addition.} Note that we write $\{e_1,\dots,e_k\}$ for the standard basis of $\N^k$.   We often call a morphism $\lambda\in\Lambda$ a \textbf{path} (or an element) in $\Lambda$, and call $\Lambda^0:= d^{-1}(0)$ the set of \textbf{vertices} of $\Lambda$; then the factorization property gives us \textbf{range} and \textbf{source} maps  $r,s:\Lambda\to \Lambda^0$. % for the range and source maps in $\Lambda$. 
For $v,w\in\Lambda^0$ and $n\in \N^k$, we write
\[
v\Lambda^nw:=\{\lambda\in\Lambda : d(\lambda)=n, r(\lambda)=v, s(\lambda)=w\}.
\]
Thus $\lambda\in v\Lambda^n w$ means that $\lambda$ is a path that starts at $w$, ends at $v$ and has  shape $n$.  Given two paths $\lambda, \nu \in \Lambda$,We can think of $\lambda$ as a $k$-cube in a $k$-colored graph %\st{determined by factorization rules}
as in the next example.

\begin{example}\label{ex:k-graph}
Consider the following two 2-colored graphs $\Gamma_1$ on the left and $\Gamma_2$ on the right. In both graphs, the dashed edges are red and the solid edges are blue.  (The sphere-like 2-graph picture below is taken from \cite{KPS-hlogy} and we would like to thank them for sharing their picture). {We will explain how $\Gamma_1, \Gamma_2$ give rise to 2-graphs $\Lambda_i$; the degree functor $d: \Lambda_i \to \N^2$ will count the number of red and blue edges in a path $\lambda \in \Lambda_i$.}
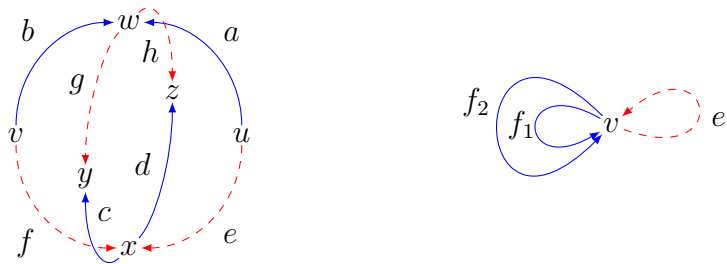
\begin{equation*}
\parbox[c]{0.9\textwidth}{\hfill
\begin{tikzpicture}[scale=1.5]
    \node[inner sep= 1pt] (100) at (1,0,0) {$u$};
    \node[inner sep= 1pt] (-100) at (-1,0,0) {$v$};
    \node[inner sep= 1pt] (010) at (0,1,0) {$w$};
    \node[inner sep= 1pt] (0-10) at (0,-1,0) {$x$};
    \node[inner sep= 1pt] (001) at (0,0,1) {$y$};
    \node[inner sep= 1pt] (00-1) at (0,0,-1) {$z$};
    \draw[-latex, blue] (100) .. controls +(0,0.6,0) and +(0.6,0,0) .. (010) node[pos=0.5, anchor=south west] {\color{black}$a$};
    \draw[-latex, red, dashed] (100) .. controls +(0,-0.6,0) and +(0.6,0,0) .. (0-10) node[pos=0.5, anchor=north west] {\color{black}$e$};
    \draw[-latex, blue] (-100) .. controls +(0,0.6,0) and +(-0.6,0,0) .. (010) node[pos=0.5, anchor=south east] {\color{black}$b$};
    \draw[-latex, red, dashed] (-100) .. controls +(0,-0.6,0) and +(-0.6,0,0) .. (0-10) node[pos=0.5, anchor=north east] {\color{black}$f$};
    \draw[-latex, red, dashed] (010) .. controls +(0,0,0.6) and +(0,0.6,0) .. (001) node[pos=0.5, anchor=east] {\color{black}$g$};
    \draw[-latex, red, dashed] (010) .. controls +(0,0,-0.6) and +(0,0.6,0) .. (00-1) node[pos=0.8, anchor=east] {\color{black}$h$};
    \draw[-latex, blue] (0-10) .. controls +(0,0,0.6) and +(0,-0.6,0) .. (001) node[pos=0.85, anchor=west] {\color{black}$c$};
    \draw[-latex, blue] (0-10) .. controls +(0,0,-0.6) and +(0,-0.6,0) .. (00-1) node[pos=0.5, anchor=south east] {\color{black}$d$};
\end{tikzpicture}
\hfill
\begin{tikzpicture}[scale=1.5]
\node[inner sep=0.5pt, circle] (v) at (0,0) {$v$};
\draw[-latex, blue] (v) edge [out=150, in=210, loop, min distance=25, looseness=2.5] (v);
\draw[-latex, blue] (v) edge [out=135, in=225, loop, min distance=50, looseness=2.5] (v);
\draw[-latex, red, dashed] (v) edge [out=-20, in=40, loop, min distance=30, looseness=2.5] (v);
\node at (-0.8, 0) {$f_1$}; \node at (-1.2, 0.2) {$f_2$}; \node at (0.95, 0.05) {$e$};
\end{tikzpicture}
\hfill \hfill }
\end{equation*}
{Depending on the choice of factorization rules, these $2$-colored graphs can give rise to several different $2$-graphs.}

There is only one 2-graph $\Lambda_1$ with the 2-colored graph $\Gamma_1$; the factorization rules of $\Lambda_1$ are given by
\[\begin{split}
&{h}{b}={d}{f},\quad {h}{a}={d}{e},\quad {g}{b}={c}{f},\quad \text{and}\quad {g}{a}={c}{e}.\\
\end{split}\]
Note that the path $hb$ has degree $e_1+e_2 = (1,1)\in \N^2$. The factorization rule $hb=df$ means that the element $hb= df$ of $\Lambda_1$ can be understood as the following square; the 2-graph $\Lambda_1$ has four such squares (paths, or elements).
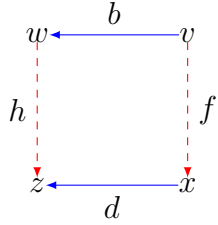
\[
\begin{tikzpicture}[yscale=0.25,xscale=0.25]
    \node[inner sep=0pt] (m) at (0,0) {$z$};
    \node[inner sep=0pt] (n) at (0,8) {$w$};
    \node[inner sep=0pt] (e) at (8,0) {$x$};
    \node[inner sep=0pt] (ne) at (8,8) {$v$};
    \draw[-latex, blue] (ne.west)--(n.east) node[pos=0.5,above] {\color{black}$b$};
    \draw[-latex, blue] (e.west)--(m.east) node[pos=0.5,below] {\color{black}$d$};
    \draw[-latex, red, dashed] (n.south)--(m.north) node[pos=0.5,left] {\color{black}$h$};
    \draw[-latex, red, dashed] (ne.south)--(e.north) node[pos=0.5,right] {\color{black}$f$};
\end{tikzpicture}
\]
However, on $\Gamma_2$, there are two 2-graphs $\Lambda_2$ and $\Lambda_3$ associated to the 2-colored graph $\Gamma_2$. The factorization rules for $\Lambda_2$ are given by
\[
f_1e=ef_1 \quad \text{and}\quad f_2e=ef_2.
\]
The factorization rules for $\Lambda_3$ are given by
\[
f_1e=ef_2\quad \text{and}\quad f_2e=ef_1.
\]
{ We leave it to the reader to check that both choices of factorization rules give rise to a well-defined functor $d: \Lambda_i \to \N^2$ satisfying the factorization property, where $d(\lambda) = (m,n)$ implies that the path $\lambda$ contains $m$ red edges and $n$ blue edges.}

\end{example}

We say that $\Lambda$ is \textbf{finite} if $\Lambda^n$ is finite for all $n\in\N^k$ and is \textbf{strongly connected} if $v\Lambda w\ne \emptyset$ for all $v,w\in \Lambda^0$. We say that $k$-graph has \textbf{no sources} if $v\Lambda^{e_i}\ne \emptyset$ for all $v\in\Lambda^0$ and for all $1\le i\le k$. Note that we only consider finite $k$-graphs with no sources in this section. Define an {\bf infinite path} in $\Lambda$ to be a  morphism  from $\Omega_k$ to $\Lambda$. To be more precise, consider the set
  \[
  \Omega_k:=\{(p,q)\in \N^k\times \N^k:p\le q\}.
  \]
  Then $\Omega_k$ is a $k$-graph with  $\Omega_k^0=\N^k$; the range and source maps $r,s:\Omega_k\to \N^k$ given by $r(p,q):=p$ and $s(p,q):=q$; and the degree functor $d$ given by $d(p,q)=q-p$. Note that the composition is given by $(p,q)(q,m)=(p,m)$ and $\Omega_k$ has no sources. An infinite path in a $k$-graph $\Lambda$ is a $k$-graph morphism $x:\Omega_k\to \Lambda$ and the infinite path space $\Lambda^\infty$ is the collection of all infinite paths.
  The space $\Lambda^\infty$ is equipped with a compact open topology generated by the cylinder sets $\{Z(\lambda):\lambda\in\Lambda\}$, where
\[
Z(\lambda)=\{x\in \Lambda^\infty : x(0, d(\lambda))=\lambda\}.
\]
For $p\in \N^k$, there is a shift map $\sigma^p$ on $\Lambda^\infty$ given by $\sigma^p(x)(m,n)=x(m+p,n+p)$ for $x\in\Lambda^\infty$.  {For more details on the above constructions, see Section 2 of \cite{KP}.}

For each $1\le i\le k$, we write $A_i$ for the \textbf{vertex matrices} for $\Lambda$, where the entries $A_i(v,w)$ are the number of paths from $w$ to $v$ with degree $e_i$. Because of the factorization property, the vertex matrices $A_i$ commute, and if $\Lambda$ is strongly connected, Lemma 4.1 of \cite{aHLRS3} establishes that there is a unique normalized
Perron-Frobenius eigenvector for the matrices $A_i$. The Perron-Frobenius eigenvector $x^\Lambda$ is the unique vector $x^\Lambda \in (0,\infty)^{|\Lambda^0|}$ with $\ell^1$-norm 1 %and all positive entries 
which is a common eigenvector of the matrices $A_i$. It is well known now (see~\cite{aHLRS3} Theorem 8.1)
that for a strongly connected finite $k$-graph $\Lambda$, there is a unique Borel probability measure $M$ on $\Lambda^\infty$, called the Perron-Frobenius measure, such that 
\begin{equation}\label{eq:measure}
M(Z(\lambda))=\rho(\Lambda)^{-d(\lambda)}x^\Lambda_{s(\lambda)}\quad \text{for all $\lambda\in\Lambda$,}
\end{equation}
where $\rho(\Lambda)=(\rho(A_1),\dots,\rho(A_k))$. % and $x^\Lambda$ is the unimodular Perron-Frobenius eigenvector of $\Lambda$. 
See \cite{aHLRS3} for the construction of the measure $M$.

For a finite $k$-graph with no sources, the Cuntz-Krieger $C^*$-algebra $C^*(\Lambda)$, often called a $k$-graph $C^*$-algebra, is a universal $C^*$-algebra generated by a collection of partial isometries $\{t_\lambda:\lambda\in\Lambda\}$ satisfying the following Cuntz-Krieger relations:
\begin{equation}\label{eq:CK1--4}\begin{split}
&\text{(CK1) $\{t_v: v \in \Lambda\z\}$ is a family of mutually orthogonal projections},\\
&\text{(CK2) $t_\mu t_\lambda = t_{\mu \lambda}$ whenever $s(\mu) = r(\lambda)$},\\
&\text{(CK3) $t_\mu^* t_\mu = t_{s(\mu)}$ for all $\mu$, and}\\
&\text{(CK4) for all $v\in\Lambda^0$ and $n\in\N^k$, we have}\quad t_v=\sum_{\lambda\in v\Lambda^n} t_\lambda t^*_\lambda.\\
\end{split}
\end{equation}
Also we can show that
\[
C^*(\Lambda)=\clsp\{t_\lambda t^*_\nu:\lambda,\nu\in \Lambda, s(\lambda)=s(\nu)\}.
\]

\subsection{$\Lambda$-semibranching function systems and representations of $C^*(\Lambda)$}
\label{sec:sbfs}
We briefly review %the definition of a semibranching function system from \cite{MP,bezuglyi-jorgensen} and 
the definition of a $\Lambda$-semibranching function system given in \cite{FGKP}, then discuss the recent works in \cite{FGKP}.

Compare the following definition with the definition of a semibranching function system given in Definition \ref{semibranchingdef}.
%% Define the functions up to sets of measure 0
%\begin{defn}
%\label{def-1-brach-system}\cite[Definition~2.1]{MP}\label{defn:sbfs}
%Let $(X,\mu)$ be a measure space, and let $I$ be a finite index set such that $\vert I\vert=N$. Suppose that, for each $i \in I$, we have a measurable map $\sigma_i:D_i\to X$, for some measurable subsets $D_i\subset X$. The family $\{\sigma_i\}$ is a \emph{semibranching function system} if the following holds.
%\begin{itemize}
%\item[(a)] There exists a corresponding family $\{R_i\}_{i=1}^N$ of subsets of $X$ with the property that
%\[
%\mu(X\setminus \bigcup_i R_i)=0\quad\text{and}\quad\mu(R_i\cap R_j)=0\;\;\text{for $i\ne j$},
%\]
%where $R_i=\sigma_i(D_i)$.
%\item[(b)] There is a Radon-Nikodym derivative
%\[
%\Phi_{\sigma_i}=\frac{d(\mu\circ\sigma_i)}{d\mu}
%\]
%with $\Phi_{\sigma_i}>0$, $\mu$-almost everywhere on $D_i$.
%\end{itemize}
%A measurable map $\sigma:X\to X$ is called a \emph{coding map} for the family $\{\sigma_i\}$ if $\sigma\circ\sigma_i(x)=x$ for all $x\in D_i$.
%\end{defn}

\begin{defn}
\label{def-lambda-brach-system}
Let $\Lambda$ be a finite $k$-graph and let $(X, \mu)$ be a measure space.  A \textbf{$\Lambda$-semibranching function system} on $(X, \mu)$ is a collection $\{D_\lambda\}_{\lambda \in \Lambda}$ of measurable subsets of $X$, together with a family of \textbf{prefixing maps} $\{\tau_\lambda: D_\lambda \to X\}_{\lambda \in \Lambda}$, and a family of coding maps $\{\tau^m: X \to X\}_{m \in \N^k}$, such that
\begin{itemize}
\item[(a)] For each $m \in \N^k$, the family $\{\tau_\lambda: d(\lambda) = m\}$ is a semibranching function system, with coding map $\tau^m$.
\item[(b)] If $ v \in \Lambda^0$, then  $\tau_v = id$,  and $\mu(D_v) > 0$.
\item[(c)] Let $R_\lambda = \tau_\lambda D_\lambda$. For each $\lambda \in \Lambda, \nu \in s(\lambda)\Lambda$, we have $R_\nu \subseteq D_\lambda$ (up to a set of measure 0), and
\[\tau_{\lambda} \tau_\nu = \tau_{\lambda \nu}\text{ a.e.}\]
 (Note that this implies that up to a set of measure 0, $D_{\lambda \nu} = D_\nu$ whenever $s(\lambda) = r(\nu)$).
\item[(d)] The coding maps satisfy $\tau^m \circ \tau^n = \tau^{m+n}$ for any $m, n \in \N^k$.  (Note that this implies that the coding maps pairwise commute.)
\end{itemize}
\end{defn}

\begin{rmk}\label{rmk:lambda-SBFS1}
\begin{itemize}
\item[(1)] The key condition of a $\Lambda$-semibranching function system is the condition (c). The immediate consequence is that $D_\lambda=D_{s(\lambda)}$ and $R_\lambda\subset R_{r(\lambda)}$ for all $\lambda\in\Lambda$. Also for $\lambda,\nu\in\Lambda$, if $s(\lambda)=r(\nu)$, then $x\in R_{\lambda\nu}$ if and only if $x\in R_\lambda$ and $\tau^{d(\lambda)}(x)\in R_\nu$.
\item[(2)] When $E$ is a finite directed graph, the definition of an $E$-semibranching function system in Definition~\ref{def-lambda-brach-system} is not equivalent to the semibranching function system of $E$ in Definition~\ref{semibranchingdef}. First of all, the set of domains $\{D_e: e\in E\}$ in Definition~\ref{semibranchingdef} neither have to be mutually disjoint nor the union to be the whole space $X$ up to a set of measure zero. But since Definition~\ref{def-lambda-brach-system}(b) requires that $D_v=R_v$ for $v\in E^0$, the condition (a) of Definition~\ref{def-lambda-brach-system} implies that $\mu(D_v\cap D_w)=\mu(R_v\cap R_w)=0$ for $v\ne w$, and $\mu(X\setminus \bigcup_{v\in E^0} R_v)=\mu(X\setminus \bigcup_{v\in E^0} D_v)=0$. As seen in Remark~\ref{rmk:lambda-SBFS1}, $D_e=D_{s(e)}$ for any $e\in E$, and hence $\mu(D_e\cap D_f) =0$ if $s(e)\ne s(f)$.

\item[(3)] It turned out that the conditions of Definition~\ref{def-lambda-brach-system} are a lot stronger than what we expected. In particular, when we have a finite  directed graph $E$, the conditions of Definition~\ref{def-lambda-brach-system} implies what is called condition (C-K) in \cite{BJ}:
\[
D_v=\bigcup_{e\in E^m} R_\lambda\quad \text{for all $v\in E^0$ and $m\in \N$}
\]
up to a measure zero set. The condition (C-K) was assumed in Theorem~2.22 in \cite{BJ} to obtain a representation of $C^*(\Lambda)$ on $L^2(X,\mu)$, where a semibranching function system is given on the measure space $(X,\mu)$.
\end{itemize}
\end{rmk}

\begin{example}\label{ex:lambda-SBFS}
Let $\Lambda$ be a strongly connected finite $k$-graph. As seen before, there is a Borel probability measure $M$ on $\Lambda^\infty$ given by the formula of \eqref{eq:measure}.  To construct a $\Lambda$-semibranching function system on $(\Lambda^\infty, M)$, we define, for $\lambda\in\Lambda$, 
 prefixing maps $\sigma_\lambda:Z(s(\lambda))\to Z(\lambda)$ by\[
\sigma_\lambda(x)=\lambda x,
\]
where we denote by $y := \lambda x$ the unique infinite path $y: \Omega_k \to \Lambda$ such that $y( (0, d(\lambda))) = \lambda$ and $\sigma^{d(\lambda)}(y) = x$.

For $m\in\N^k$ we define the coding maps $\sigma^m:\Lambda^\infty \to \Lambda^\infty$ by
\[
\sigma^m(x)=x(m,\infty).
\]
Then $\{\sigma_\lambda\}_{\lambda\in\Lambda}$ with $\{\sigma^m\}_{m\in\N^k}$ form a $\Lambda$-semibranching function system on $(\Lambda^\infty, M)$ as shown in Proposition~3.4 of \cite{FGKP}.

\end{example}

When a $k$-graph $\Lambda$ is finite and has no sources, one of the main theorems of \cite{FGKP}, Theorem~3.5, says that the operators $S_\lambda$ associated to a $\Lambda$-semibranching function system on a measure space $(X,\mu)$ given by
\begin{equation}\label{eq:S_lambda}
S_\lambda \xi(x):=\chi_{R_\lambda}(x)(\Phi_{\tau_\lambda}(\tau^{d(\lambda)}(x)))^{-1/2}\xi(\tau^{d(\lambda)}(x))
\end{equation}
generate a representation of $C^*(\Lambda)$ on $L^2(X,\mu)$, where 
\[\Phi_{\tau_\lambda} = \frac{d(\mu \circ \tau_\lambda)}{d\mu}\]
is positive a.e. on $D_\lambda$.

In addition, if we have a strongly connected finite $k$-graph $\Lambda$, then the $\Lambda$-semibranching function system of Example \ref{ex:lambda-SBFS} on the Borel probability measure space $(\Lambda^\infty, M)$ %discussed in the previous subsection, then the representation given by \eqref{eq:S_lambda} on $L^2(\Lambda^\infty, M)$ 
gives rise to a representation of $C^*(\Lambda)$ on $L^2(\Lambda^\infty, M)$ which is faithful if and only if $\Lambda$ is aperiodic. (See Theorem~3.6 of \cite{FGKP}).

Moreover, if the vertex matrices $A_i$ associated to a strongly connected finite $k$-graph $\Lambda$ %have only $0$ and $1$ entries, 
are all $\{0,1\}$-matrices, then we can construct $\Lambda$-semibranching function systems on a fractal subspace $X$ of $[0,1]$. % with Hausdorff measure $\mu$ with its Hausdorff dimension $s$. 
In particular, let $N=|\Lambda^0|$ and label the vertices of $\Lambda$ by the integers, $0, 1,\dots, N-1$. Let $\rho(A)$ denote the spectral radius %Perron-Frobenius eigenvector 
of the product $A:=A_1\dots A_k$. Then consider the embedding $\Psi: \Lambda^\infty\to [0,1]$ given by interpreting the sequence of vertices of a given infinite path as an  $N$-adic decimal.  %expansions of the sequence of vertices of a given infinite path. 
Then $X=\Psi(\Lambda^\infty)$ is a Cantor-type fractal subspace of $[0,1]$ and the Hausdorff measure $\mu$ on $X$ is given by the Borel probability measure $M$ on $\Lambda^\infty$ via $\Psi$. The prefixing maps $\{\tau_\lambda\}$ and coding maps $\{\tau^{d(\lambda)}\}$ on $(X,\mu)$ are induced from the prefixing maps $\{\sigma_\lambda\}$ and coding maps $\{\sigma^m\}$ on $(\Lambda^\infty, M)$ given in Example~\ref{ex:lambda-SBFS}. Moreover, if $s$ denotes the Hausdorff dimension of $X$, we have
\[
N^{ks}=\rho(A),\quad\text{and}\quad s=\frac{1}{k}\frac{\ln \rho(A)}{\ln N}.
\]
See Section~3.2 of \cite{FGKP} for further details.

%Moreover, KMS states associated to $\Lambda$-semmibranching function system were discussed in section~3.3 of \cite{FGKP} when the associated Radon-Nikodym derivatives are constant but we will not discuss this matter here.

\section{Wavelets on $L^2(\Lambda^\infty, M)$}
\label{sec:wavelets}

Let $\Lambda$ be a strongly connected finite $k$-graph. As seen in the previous section, there is a Borel probability measure $M$ on the infinite path space $\Lambda^\infty$ given by, for $\lambda\in \Lambda$,
\[
M(Z(\lambda))=\rho(\Lambda)^{-d(\lambda)} x^\Lambda_{s(\lambda)},
\]
where $\rho(\Lambda)=(\rho(A_i))_{1\le i\le k}$ and $x^\Lambda$ is the unimodular Perron-Frobenius eigenvector of $\Lambda$.
We now proceed to generalize the wavelet decomposition of $L^2(\Lambda^\infty, M)$ that we constructed in Section 4 of \cite{FGKP}.  In that paper, we built an orthonormal decomposition of $L^2(\Lambda^\infty, M)$, which we termed a \textbf{wavelet decomposition}, following Section 3 of \cite{MP}.
Here,  our wavelet decomposition is constructed by applying (some of) the operators $S_\lambda$ of %Section 3 of \cite{FGKP} 
Example \ref{ex:lambda-SBFS} and Equation \eqref{eq:S_lambda}
to a basic family of functions in $L^2(\Lambda^\infty, M)$.  Instead of choosing the finite paths $\lambda$ whose degrees are associated to $k$-cubes, we will construct them from isometries given by paths whose degrees are given by $k$-rectangles.   One way to interpret our main result below (Theorem \ref{Wavelets-Theo}) is to say that for any rectangle $(j_1, j_2, \ldots, j_k) \in \N^k$ with no zero entries, the cofinal set $\{n \cdot (j_1, j_2, \ldots, j_k): n \in \N\} \subseteq \N^k$ gives rise to an orthonormal decomposition of $L^2(\Lambda^\infty, M)$.

While we can use the same procedure to obtain a family of orthonormal functions in $L^2(X, \mu)$ whenever we have a $\Lambda$-semibranching function system on $(X, \mu)$, we cannot establish in general that this orthonormal decomposition densely spans $L^2(X, \mu)$ -- we have no analogue of Lemma \ref{lem:squares-span} for general $\Lambda$-semibranching function systems.  Moreover, by Corollary 3.12 of \cite{FGKP}, every $\Lambda$-semibranching function system  on $\Lambda^{\infty}$ with constant Radon-Nikodym derivative is endowed with the Perron-Frobenius measure. Thus, in this section, we restrict ourselves to the case of $(\Lambda^\infty, M)$.  We also note that our proofs in this section follow the same ideas found in the proof of Theorem 3.2  of \cite{MP}.

For a path $\lambda \in \Lambda$, let $\Theta_\lambda$ denote the characteristic function of $Z(\lambda) \subseteq \Lambda^\infty$.  Recall that $M$ is the unique Borel probability measure on $\Lambda^\infty$ satisfying our desired properties. For the rest of this section, we fix a $k$-tuple
$$(j_1,j_2,\cdots, j_k)\in\mathbb N^k$$
 all of whose coordinates are {\em positive} integers.

\begin{lemma}
\label{lem:squares-span}
Let $\Lambda$ be a strongly connected $k$-graph and fix $J = (j_1, j_2, \ldots, j_k) \in (\Z^+)^k$.
Then the span of the set
\[ S^J := \{\Theta_\lambda: d(\lambda) = (n\cdot j_1, n\cdot j_2, \ldots, n\cdot j_k) \text{ for some } n \in \N\}\]
is dense in $L^2(\Lambda^\infty, M)$.
\label{dense-squares}
\end{lemma}
\begin{proof}
Let $\mu \in \Lambda$.  We will show that we can write $\Theta_\mu$ as a linear combination of functions from $S^J$.

Suppose $d(\mu) = (m_1, \ldots, m_k)$.  Let $m = \min\{N>0: N\cdot j_i-m_i\geq 0 \;\;\text{for}\;\; 1\leq i\leq k \},$ and let $n = (m\cdot j_1, m\cdot j_2, \ldots, m\cdot j_k) - d(\mu)$.  Let
\[C_\mu = \{\lambda \in \Lambda: r(\lambda) = s(\mu), d(\lambda) = n\}.\]
  In words, $C_\mu$ consists of the paths that we could append to $\mu$ such that $\mu \lambda \in S^J$: if $\lambda \in C_\mu$ then the product $\mu \lambda$ is defined and \[d(\mu \lambda) = d(\mu) + d(\lambda) = (m\cdot j_1, m\cdot j_2, \ldots, m\cdot j_k).\]

 Similarly, since $d(\mu \lambda) = d(\mu \lambda') = (m\cdot j_1, \ldots, m\cdot j_k) = mJ$, if $x \in Z(\mu \lambda) \cap Z(\mu \lambda')$ then the fact that $x(0, mJ)$ is well defined implies that
\[ x(0, mJ) = \mu \lambda = \mu \lambda' \Rightarrow \lambda = \lambda'.\]
It follows that if $\lambda\ne  \lambda' \in C_\mu$, then $Z(\mu \lambda) \cap Z(\mu \lambda') = \emptyset$.  Since every infinite path $x\in Z(\mu)$ has a well-defined ``first segment'' of shape $(m\cdot j_1, \ldots, m\cdot j_k)$ -- namely $x(0, mJ)$ -- every $x \in Z(\mu)$ must live in $Z(\mu \lambda)$ for precisely one $\lambda \in C_\mu$.  Thus, we can write $Z(\mu)$ as a disjoint union,
\[Z(\mu) = \bigsqcup_{\lambda \in C_\mu} Z(\mu \lambda).\]

It follows that $\Theta_\mu = \sum_{\lambda \in C_\mu} \Theta_{\mu \lambda}$, so the span of functions in $S^J$ includes the characteristic functions of cylinder sets.  Since the cylinder sets $Z(\mu)$ form a basis for the topology on $\Lambda^\infty$ with respect to which $M$ is a Borel measure, it follows that the span of $S^J$ is dense in $L^2(\Lambda^\infty, M )$ as claimed.
\end{proof}

Since the span of the functions in $S^J$ is dense in $L^2(\Lambda^\infty, M)$, we will show how to decompose $\overline{\sp}\; S^J$ as
an orthogonal direct sum,
\[\overline{\sp}\; S^J = \mathcal{V}_{0,\Lambda}\oplus\bigoplus_{j=0}^\infty \mathcal{W}^J_{j,\Lambda},\]
where  $\mathcal{V}_{0,\Lambda}$ will be equal to the subspace spanned by the functions $\{\Theta_v: v \in \Lambda\z\}.$ 
We then will construct $\mathcal{W}^J_{j,\Lambda}$  for each $j > 1$ from the functions in $\mathcal{W}^J_{0,\Lambda}$ and (some of) the operators $S_\lambda$ discussed in Section 3 of \cite{FGKP}. The construction of $\mathcal{W}^J_{0,\Lambda}$ generalizes that given in Section 4 of \cite{FGKP}, which in turn was similar to that given in Section 3 of \cite{MP} for the case of {a directed graph}.

We recall from \cite{FGKP} that the functions $\{\Theta_v: v \in \Lambda\z\}$ form an orthogonal set in $L^2(\Lambda^\infty, M),$ whose span includes those functions that are constant on $\Lambda^{\infty}$:
\begin{align*}
\int_{\Lambda^\infty} \Theta_v \overline{\Theta_w} \, dM &= \delta_{v,w} M(Z(v)) = \delta_{v,w} x^\Lambda_{v},
\end{align*}
and $$\sum_{v\in \Lambda\z}\Theta_v(x)\equiv 1.$$
Thus,  the set $\{\frac{1}{\sqrt{x^\Lambda_v}} \Theta_v: v\in \Lambda\z\}$ is an orthonormal set in $S^J.$  We define
\[\mathcal{V}_{0,\Lambda}:= \overline{\sp}\{\frac{1}{\sqrt{x^\Lambda_v}} \Theta_v: v\in \Lambda\z\}.\]
% Here, in general, we'd have to use \frac{1}{\sqrt{\mu(D_v)}}.  Do we know this is always positive?  I think this might follow from the fact that the Radon-Nikodym derivative is positive.

To construct $\mathcal{W}^J_{0,\Lambda}$, let $v \in \Lambda\z$ be arbitrary.  Let
\[D_v^J = \{\lambda\in \Lambda : d(\lambda) = J \text{ and } r(\lambda) = v\},\]
and write $d_v^J$ for $|D_v^J|$ (note that by our hypothesis that $\Lambda$ is a finite $k$-graph we have $d_v^J < \infty$).

Define an inner product on $\C^{d_v^J}$ by
\begin{equation}
\label{inner-prod}
\l \vec{v}, \vec{w} \r = \sum_{\lambda \in D_v^J}  \overline{v_\lambda} w_\lambda \rho(\Lambda)^{(-j_1, \ldots, -j_k)} x^\Lambda_{s(\lambda)},\end{equation}
and let $\{ c^{m, v}\}_{m=1}^{d_v^J - 1}$ be an orthonormal basis for the orthogonal complement of $(1, \ldots, 1) \in \C^{d_v^J}$ with respect to this inner product.
Let $c^{0,v}$ be the unique vector of norm one with respect to this inner product with (equal) positive entries that is a multiple of $(1, \ldots, 1) \in \C^{d_v^J}.$
Thus, $\{ c^{m, v}\}_{m=0}^{d_v - 1}$ is an orthonormal basis for $\C^{d_v^J}.$ %whose $d_v-1$ final vectors form an orthonormal basis for the orthogonal complement of $(1, \ldots, 1) \in \C^{d_v^J}$ with respect to our specialized inner product.

We explain the importance of $(1,\ldots,1)\in \C^{d_v^J}$ further. We index the $\lambda$'s in $D_v^J:$
$$D_v^J= \{\lambda_1, \lambda_2, \cdots \lambda_{d_v^J}\}.$$
We need to stress here that $$\sum_{j=1}^{d_v^J}\Theta_{\lambda_j} = \Theta_v.$$

In this way, we have identified $\Theta_v$ with $(1,1,\cdots, 1) \in \mathbb C^{d_v^J}.$
(When we do this, we identify $\Theta_{\lambda_1}$ with $(1,0,0,\cdots, 0), \Theta_{\lambda_2}$ with $(0,1,0,\cdots, 0),$ and $\Theta_{\lambda_{d_v^J}}$ with $(0,0,0,\cdots, 1) \in \mathbb C^{d_v^J}.$)

Now, for each pair $(m,v)$ with $0\leq m \leq d_v^J - 1$ and $v$ a vertex in $\Lambda\z$, define
\[f^{m, v} = \sum_{\lambda \in D_v^J} c^{m,v}_\lambda \Theta_\lambda.\]
Note that by our definition of the measure  $M$ on $\Lambda^\infty$, since for $1\leq m\leq d_v^J-1,$ the vectors $c^{m,v}$ are orthogonal to $(1, \ldots, 1)$ in the inner product \eqref{inner-prod}, we have
\begin{align*}
\int_{\Lambda^\infty} f^{m,v} \, dM & = \sum_{\lambda \in D_v^J} c^{m,v}_\lambda M(Z(\lambda)) \\
&= \sum_{\lambda \in D_v^J} c^{m,v}_\lambda \rho(\Lambda)^{(-j_1, \ldots, j_k)} x^\Lambda_{s(\lambda)}\\
&= 0
\end{align*}
for each $(m,v)$ with $m\geq 1.$  On the other hand, if $m=0,$ it is easy to see that
\[ f^{0,v}= \sum_{\lambda \in D_v^J} c^{0,v}_\lambda \Theta_\lambda\]
is a constant multiple of $\Theta_v,$ since %the orthogonality of $\{c^{m,v}\}$ implies 
$c^{0,v}_\lambda=c^{0,v}_{\lambda'}$ for $\lambda,\;\lambda'\in \; D_v^J,$ and $\sum_{\lambda \in D_v^J}\Theta_\lambda=\Theta_v.$
Moreover, the arguments of Lemma \ref{dense-squares} tell us that $\Theta_\lambda \Theta_{\lambda'} = \delta_{\lambda, \lambda'} \Theta_\lambda$ for any $\lambda, \lambda'$ with $d(\lambda) = d(\lambda') = (j_1, \ldots, j_k)$.  Consequently, if $\lambda\in D_v^J, \lambda' \in D_{v'}^J$ for $v \not= v'$, we have  $\Theta_\lambda \Theta_{\lambda'} =0$.  It follows that
\begin{align*}
\int_{\Lambda^\infty} f^{m,v} \overline{f^{m',v'}} \, dM
&= \delta_{v, v'} \sum_{\lambda\in D_v^J} c^{m,v}_\lambda \overline{c^{m',v}_\lambda} M(Z(\lambda)) \\
&= \delta_{v, v'} \delta_{m,m'}
\end{align*}
since the vectors $\{c^{m,v}\}$ form an orthonormal set with respect to the inner product \eqref{inner-prod}.
Thus, the functions $\{f^{m,v}\}$ are an orthonormal set in $L^2(\Lambda^\infty, M)$.  We define
\[\mathcal{W}^J_{0,\Lambda} :=\overline{ \sp} \{f^{m,v}: v \in \Lambda\z, 1 \leq m \leq d_v^J-1\}.\]

Note that $\mathcal{V}_{0,\Lambda}$  is orthogonal to $\mathcal{W}^J_{0,\Lambda}$.  To see this, let $ g \in \mathcal{V}_{0,\Lambda}$  be arbitrary, so $g = \sum_{v\in \Lambda\z} g_v \Theta_v$ with $g_v \in \C$ for all $v$. Then
\begin{align*}
\int_{\Lambda^\infty} \overline{f^{m,v'}(x)} g(x) \, dM &= \sum_{v\in V_0}\delta_{v', v}\, g_v \sum_{\lambda\in D_{v'}^J} \overline{c^{m,v'}_\lambda} M(Z(\lambda)) \\
&= 0,
\end{align*}
since $\sum_{\lambda\in D_v^J} c^{m,v}_\lambda M(Z(\lambda)) = 0$ for all fixed $v,$ and $1\leq m\leq d_v^J-1$.  Thus, $g$ is orthogonal to every basis element $f^{m,v}$ of $\mathcal{W}^J_{0,\Lambda}$.

The basis $\{f^{m,v}:\; v \in \Lambda\z, 1 \leq m \leq d_v^J-1\}$ for $\mathcal{W}^J_{0,\Lambda}$ generalizes the analogue for $k$-graphs of the  {\bf graph wavelets} of \cite{MP}, as described in Section 4 of \cite{FGKP}.
As the following Theorem shows, by shifting these functions using the operators
\[\{S_\lambda: d(\lambda) = nJ \text{ for some } n \in \N\},\] % generalizing those described in Theorem 3.5 of \cite{FGKP},
 we obtain an orthonormal basis for $L^2(\Lambda^\infty, M)$.  Thus, each $J \in (\mathbb Z^+)^k$ gives a different family of $k$-graph wavelets associated to the representation of $C^*(\Lambda)$ described in Theorem 3.5 of \cite{FGKP}.

\begin{thm}
\label{Wavelets-Theo}  (Compare to Theorem 4.2 of \cite{FGKP}) Let $\Lambda$ be a strongly connected finite $k$-graph and fix $J \in (\mathbb Z^+)^k$.
For each fixed $j \in \N^+$ and $v\in\Lambda^0,$ let
\[C_{j,v}^J:= \{\lambda \in \Lambda: s(\lambda)=v, d(\lambda) = j J\},\]
 and let $S_\lambda$ be the operator on $L^2(\Lambda^\infty, M)$ described in Theorem 3.5 of \cite{FGKP};
 for $\xi\in L^2(\Lambda^\infty, M)$,
 \[
 S_\lambda \xi(x)=\Theta_\lambda(x) \rho(\Lambda)^{d(\lambda)/2}\xi(\sigma^{d(\lambda)}(x)).
 \]
Then
\[\{S_\lambda f^{m,v}: v \in \Lambda^0, \;\lambda \in C_{j,v}^J,\; 1 \leq m \leq d_v^J-1\}\]
is an orthonormal set.  Moreover, if $\lambda \in C_{j,v}^J,\; \mu \in C_{i,v'}^J$ for $0<i<j,$ we have
\[ \int_{\Lambda^\infty} S_\lambda f^{m,v} \overline{S_\mu f^{m',v'}} \, dM = 0 \;\;\text{for $1\le m, m'\le d_v^J-1$}.\]

It follows that defining
\[{\mathcal W}^J_{j,\Lambda} := \clsp\{S_\lambda f^{m,v}: v \in \Lambda^0, \lambda \in C_{j,v}^J, 1 \leq m \leq d_v^J-1\},\]
for $j \geq 1$, we obtain an orthonormal decomposition
\[L^2(\Lambda^\infty, M) = \clsp\;S^J = {\mathcal V_{0,\Lambda}}\oplus \bigoplus_{j=0}^\infty {\mathcal W}^J_{j,\Lambda}.\]
\end{thm}

\begin{proof}
We first observe that if $s(\lambda) = v$, then % $S_\lambda f^{m,v}$ is nonzero precisely when $s(\lambda) = v$, and in this case
\[S_\lambda f^{m,v} = \sum_{\mu \in D_v^J} c^{m,v}_\mu \rho(\Lambda)^{d(\lambda)/2} \Theta_{\lambda \mu} ,\]
because the Radon-Nikodym derivatives $\Phi_{\sigma_\lambda}$ are constant on $Z(s(\lambda))$ for each $\lambda \in \Lambda$, thanks to  Proposition 3.4 of \cite{FGKP}.
In particular, if $d(\lambda) = 0$ then $S_\lambda f^{m,v} = f^{m,v}$.
Thus, if $d(\lambda) = d( \lambda') = (j\cdot j_1, \ldots, j\cdot j_k)$, the factorization property and the fact that $d(\lambda \mu ) = d(\lambda'\mu') = ((j+1)\cdot j_1, \ldots, (j+1)\cdot j_k)$ for every $\mu \in D_{s(\lambda)}^{J}, \mu' \in D_{s(\lambda')}^{J}$ implies that
\[\Theta_{\lambda \mu} \Theta_{\lambda' \mu'} = \delta_{\lambda, \lambda'} \delta_{\mu,\mu'}\;\;\text{for all $\mu \in D_{s(\lambda)}^J, \mu' \in D_{s(\lambda')}^J$}.\]
In particular, $S_\lambda f^{m, v} \overline{S_{\lambda'}f^{m',v'}} = 0$ unless $\lambda = \lambda'$ (and hence $v = v'$).  Moreover,
\begin{align*}
\int_{\Lambda^\infty} S_\lambda f^{m,v} \overline{S_\lambda f^{m',v}} \, dM &=
\sum_{\mu \in D_v^J} c^{m,v}_\mu \overline{c^{m',v}_\mu}\rho(\Lambda)^{d(\lambda)} M(Z(\lambda \mu))\\
&= \sum_{\mu \in D_v^J} c^{m,v}_\mu \overline{c^{m', v}_\mu}  \rho(\Lambda)^{-d(\mu)} x^{\Lambda}_{s(\mu)}\\
&=   \delta_{m, m'},
\end{align*}
by the definition of the vectors $c^{m,v}_\mu$, since $d(\mu) = (j_1,j_2,\ldots, j_k)$ for each $\mu \in D_v^J$.

Now, suppose $\lambda\in C_{1,v}^J$.  Observe that $S_\lambda f^{m,v} \overline{f^{m',v'}}$ is nonzero only when $v' = r(\lambda)$, and also that
\[
(S_\lambda f^{m,v})(x) \overline{f^{m',v'}}(x)=\sum_{\mu\in D_v^J} c^{m,v}_\mu \rho(\Lambda)^{d(\lambda)/2}\Theta_{\lambda\mu}(x)\sum_{\mu'\in D_{v'}^J}\overline{c^{m',v'}_{\mu'}}\Theta_{\mu'}(x).
\]
Note that $\Theta_{\lambda\mu}(x)\Theta_{\mu'}(x)\ne 0$ if $x=\lambda\mu y=\mu'y'$ for some $y,y'\in\Lambda^\infty$, and $\lambda\in C_{1,v}^J$ implies $d(\lambda)=J=d(\mu')$. So the factorization property implies that $\mu'=\lambda$, and hence we obtain
\begin{align*}
\int_{\Lambda^\infty}S_\lambda f^{m,v} \overline{f^{m',v'}} \, dM &=  \sum_{\mu \in D_v^J} c^{m,v}_\mu \overline{c^{m',v'}_\lambda} \rho(\Lambda)^{d(\lambda)/2} M(Z(\lambda\mu)) \\
&= \overline{c^{m',v'}_\lambda} \rho(\Lambda)^{-d(\lambda)/2} \sum_{\mu \in D_v^J}c^{m,v}_\mu \rho(\Lambda)^{-d(\mu)} x^\Lambda_{s(\mu)} \\
&= 0.
\end{align*}
Thus, ${\mathcal W}^J_{0,\Lambda}$ is orthogonal to ${\mathcal W}^J_{1,\Lambda}$.

In more generality, suppose that $\lambda \in C_{j,v}^J, \lambda' \in C_{i,v'}^J,\;j>i\geq 1$.  We observe that $S_\lambda f^{m,v} \overline{S_{\lambda'} f^{m', v'}}$ is nonzero only when $\lambda = \lambda' \nu$ with  $\nu \in C_{j-i,v}^J$, so we have
$$S_\lambda f^{m,v} \overline{S_{\lambda'} f^{m', v'}} =S_{\lambda'}(S_{\nu}f^{m,v})\overline{S_{\lambda'} f^{m', v'}}.$$
Consequently,
\begin{align*}
\int_{\Lambda^\infty} S_\lambda f^{m,v} \overline{S_{\lambda'} f^{m', v'}} \, dM &= \int_{\Lambda^\infty}S_{\lambda'}(S_{\nu}f^{m,v})\overline{S_{\lambda'} f^{m', v'}}\,dM\\
&= \int_{\Lambda^\infty}(S_{\nu}f^{m,v})\overline{S_{\lambda'}^*S_{\lambda'} f^{m', v'}}\,dM\\
&= \int_{\Lambda^\infty}(S_{\nu}f^{m,v})\overline{f^{m', v'}}\,dM\\
&=  \sum_{\mu \in D_v^J} c^{m,v}_\mu \overline{c^{m',v'}_\nu} \rho(\Lambda)^{d(\nu)/2} M(Z(\nu\mu)) \\
&= \overline{c^{m',v'}_\nu} \rho(\Lambda)^{-d(\nu)/2} \sum_{\mu \in D_v^J}c^{m,v}_\mu \rho(\Lambda)^{-d(\mu)} x^\Lambda_{s(\mu)} \\
&= 0.
\end{align*}
Thus, the sets ${\mathcal W}^J_{j,\Lambda}$ are mutually orthogonal as claimed.

We now need to show that $L^2(\Lambda^{\infty},M) = \mathcal{V}_{0,\Lambda}\oplus\bigoplus_{j=0}^\infty \mathcal{W}^J_{j,\Lambda}.$
We will do this by showing that
$$S^J\subset \mathcal{V}_{0,\Lambda}\oplus\bigoplus_{j=0}^\infty \mathcal{W}^J_{j,\Lambda}.$$
%This will take some doing.
We first note that if $\lambda\in\Lambda$ and $d(\lambda)=(j_1,j_2,\cdots,j_k),$ then $\Theta_{\lambda}\in \mathcal{V}_{0,\Lambda}\oplus\bigoplus_{j=0}^\infty \mathcal{W}^J_{j,\Lambda}.$    Let $r(\lambda)=v,$ so that $\lambda\in D_v^J.$
Write $\lambda=\lambda_i$ for some specific $i\in \{1,2,\cdots,d_v^J\}.$ We identify $\Theta_{\lambda_i}$ with $(0,0,\ldots , 1 \text{(in}\;i_{th}\;\text{ spot)}, 0,0, \ldots, 0)=e_i \in \C^{d_v^J}.$

As we observed above, identifying $\Theta_{\lambda_i}$ with $e_i$ induces an isomorphism between the (finite-dimensional) Hilbert spaces
$$\text{span}\{\Theta_{\lambda_1},\;\Theta_{\lambda_2},\;\cdots, \Theta_{\lambda_{d_v^J}}\}\subset L^2(\Lambda^{\infty}, M)$$
and $\mathbb C^{d_v^J}$ equipped with the  inner product \eqref{inner-prod}.  By using this isomorphism, we can identify the function $f^{m,v} =\sum_{i=1}^{d_v^J} c^{m,v}_{\lambda_i}\Theta_{\lambda_i},$
 with the vector $(c^{m,k}_{\lambda_i})_i \in \C^{d_v^J}$. This identification allows us to write
\[ \Theta_{\lambda_i} = C \langle \Theta_{\lambda_i}, c^{0,v} \rangle \Theta_v + \sum_{m=1}^{d_v^J-1} \langle \Theta_{\lambda_i}, f^{m,v} \rangle f^{m,v} \] 
for some $C \in \C$, using the orthonormality of the basis $\{c^{m,v}\}_{m=0}^{d_v^j-1}$.  In other words, $\Theta_{\lambda_i} \in \mathcal{V}_{0,\Lambda} \oplus \mathcal{W}_{0, \Lambda}^J$. It follows that $\Theta_{\lambda}\in {\mathcal V}_{0,\Lambda}\oplus {\mathcal W}^J_{0,\Lambda}$ for all $\lambda\in\Lambda$ such that $d(\lambda)=(j_1,j_2,\cdots, j_k).$

We now assume that for $1\leq j\leq m,$ if $\lambda\in \Lambda$ and
$d(\lambda)=j J,$
then for any vertex $w \in \Lambda^0$,
\begin{equation}
S_{\lambda}(\Theta_w)\;\in  {\mathcal V}_{0,\Lambda}\oplus \bigoplus_{j=0}^{m-1}{\mathcal W}^J_{j,\Lambda};\;\text{and}
\end{equation}
\begin{equation}
\Theta_{\lambda}\in {\mathcal V}_{0,\Lambda}\oplus \bigoplus_{j=0}^{m-1}{\mathcal W}^J_{j,\Lambda}.
\end{equation}
We have already established the base case $m=1.$

Let us use induction to show that if $\lambda_0\in \Lambda$ and $d(\lambda_0)=(m+1)J,$ then
$$\Theta_{\lambda_0}\in {\mathcal V}_{0,\Lambda}\oplus \bigoplus_{j=0}^{m}{\mathcal W}^J_{j,\Lambda}, \text{ and } S_{\lambda_0}(\Theta_w)\;\in  {\mathcal V}_{0,\Lambda}\oplus \bigoplus_{j=0}^{m}{\mathcal W}^J_{j,\Lambda}.$$

Fix a vertex $w\in \Lambda^0.$
Let us calculate, using our standard formulas for our representation of $C^{\ast}(\Lambda)$ on $L^2(\Lambda^{\infty},M),$
$$S_{\lambda_0}(\Theta_w(x))= \Theta_{\lambda_0}(x)(\rho(\Lambda)^{d(\lambda_0)/2})\Theta_w(\sigma^{d(\lambda_0)}(x)).$$
We first note:  for this to have any chance of being non-zero, we need $x\in Z({\lambda_0})$ and $\sigma^{d(\lambda_0)}(x)$ must be in $Z(w)$.
%{\color{blue} Elizabeth would delete this sentence for the published version: i.e. it must be the case that $x$ has $\lambda_0$ in its first $((m+1)\cdot j_1,(m+1)\cdot j_2,\cdots, (m+1)\cdots j_k)$ (repeated $k$ times) block, and then, after one ``lops off'' $\lambda_0$ from $x,$ the range of $\sigma^{d(\lambda_0)}(x)$ must be equal to $w.$ }
In other words, $s(\lambda_0) = w$.  So we obtain:
 $S_{\lambda_0}(\Theta_w)$ is a constant multiple of $\Theta_{\lambda_0}$ if $w=s(\lambda_0),$ and $S_{\lambda_0}(\Theta_w)=0$ if $w\not=s(\lambda_0).$

So, assuming that $w=s(\lambda_0),$ we have that $S_{\lambda_0}(\Theta_w)$ is a constant multiple of $\chi_{Z({\lambda_0})}=\Theta_{\lambda_0}.$ Using the factorization property, now write $\lambda_0=\lambda_1\lambda_2$ with  $s(\lambda_2)\;=\;s(\lambda_0)\;=w$ and
$$d(\lambda_1)\;=\;(j_1,j_2,\cdots, j_k)$$
and
$$d(\lambda_2)\;=\;(m\cdot j_1, m\cdot j_2,\cdots, m\cdot j_k).$$
Recall
$$S_{\lambda_0}\;=\;S_{\lambda_1\lambda_2}\;=\;S_{\lambda_1}S_{\lambda_2}.$$
By our induction hypothesis,
$$S_{\lambda_2}(\Theta_w)\in {\mathcal V}_{0,\Lambda}\oplus \bigoplus_{j=0}^{m-1}{\mathcal W}^J_{j,\Lambda}.$$
Therefore we can write
$$S_{\lambda_2}(\Theta_w)=g_0+\sum_{j=0}^{m-1}h_j,$$
where $g_0\in {\mathcal V}_{0,\Lambda}$ and $h_j\in{\mathcal W}^J_{j,\Lambda}$  for $0\leq j\leq m-1.$
So,
$$S_{\lambda_0}(\Theta_w)=S_{\lambda_1}\Big(g_0+\sum_{j=0}^{m-1}h_j\Big)=\;S_{\lambda_1}(g_0)+\sum_{j=0}^{m-1}S_{\lambda_1}(h_j).$$

We have proved directly that $S_{\lambda_1}(g_0)\in {\mathcal V}_{0,\Lambda}\oplus {\mathcal W}_{0,\Lambda},$ and
it follows from the definition of ${\mathcal W}^J_{j,\Lambda}$ that
$$S_{\lambda_1}(h_j)\in {\mathcal W}^J_{j+1,\Lambda}\;\;\text{for}\;\;0\leq j\leq m-1.$$
It follows that
$$S_{\lambda_0}(\Theta_w)\in{\mathcal V}_{0,\Lambda}\oplus \bigoplus_{j=0}^{m}{\mathcal W}^J_{j,\Lambda}.$$
Since $S_{\lambda_0}(\Theta_{s(\lambda_0)})$ is a constant multiple of $\Theta_{\lambda_0},$ we have that
$$\Theta_{\lambda_0}\in {\mathcal V}_{0,\Lambda}\oplus \bigoplus_{j=0}^{m}{\mathcal W}^J_{j,\Lambda},$$ as desired.
It follows that the spanning set
$${S^J}\subset \mathcal{V}_{0,\Lambda}\oplus\bigoplus_{j=0}^\infty \mathcal{W}^J_{j,\Lambda},$$
and thus by Lemma \ref{lem:squares-span},
$$L^2(\Lambda^{\infty},M)\;=\;\mathcal{V}_{0,\Lambda}\oplus\bigoplus_{j=0}^\infty \mathcal{W}^J_{j,\Lambda}.$$
\end{proof}

We now partially answer a question posed by A. Sims, who asked about the importance of the shape $(j,j, \ldots, j)$ of the ``cubical wavelets'' introduced in \cite{FGKP}.  As we have now shown, we can construct wavelets of any non-trivial rectangular shape, not only cubes.  Sims also  asked if there
was a relationship between the dimension of the spaces $\mathcal{W}^J_{j, \Lambda}$ and the fixed rectangular shape $J =(j_1, \ldots, j_k)$. The answer is ``Not necessarily."  We recall that for $v\in \Lambda^0,$
\[D_v^J = \{\lambda\in \Lambda : d(\lambda) = (j_1,j_2,\cdots, j_k) \text{ and } r(\lambda) = v\},\]
and $d_v^J=|D_v^J|.$   The dimension of the wavelet space ${\mathcal W}^J_{j,\Lambda}$ is equal to
$$\sum_{v\in \Lambda^0}(d_v^J-1).$$
Since each $d_v^J$ depends on both $v\in\Lambda_0$ and $(j_1,j_2,\cdots, j_k)\in [\mathbb N^+]^k$,  the dimensions obviously could change with different choices of degrees.  On the other hand, if you take a degree that is $\ell$ times another degree $(j_1,j_2,\cdots,j_k)$ , it would be interesting to check whether or not the wavelet space of level $0$ corresponding to $\ell J$, $\mathcal{W}^{\ell J}_{0, \Lambda} $, is equal to
$$\bigoplus_{j=0}^{\ell-1}{\mathcal W}^J_{j,\Lambda}.$$

We also observe that,  since $j_i \geq 1 \ \forall \ i$, the factorization property implies that every $\lambda \in D^J_v$ is associated to $\lambda_1 \in D^{(1,1,\ldots, 1)}_v$, namely, $\lambda = \lambda_1 \nu$.  In other words, $\lambda_1 = \lambda(0, (1,\ldots, 1))$ is the initial segment of $\lambda$ of shape $(1,\ldots, 1)$.  Thus, $d^J_v \geq d^{(1,\ldots, 1)}_v$ for all $v.$  In fact, by mapping the basis vector $\Theta_\mu \in \C^{d^{(1,\ldots, 1)}_v}$ to the vector 
\[\Psi_\mu = \sum_{\nu: \mu \nu \in D^{J}_v} \Theta_{\mu \nu} \in \C^{d^J_v}\]
we can transfer our orthonormal basis $\{c^{m,v}\}_m$ for $\C^{d^{(1,\ldots, 1)}_v}$ to an orthonormal set in $\C^{d^J_v}$; then we can complete this orthonormal set to form the orthonormal basis for $\C^{d^J_v}$ that we use to construct the wavelet functions $f^{m,v}$.

In other words, whenever $J\geq (1,1, \ldots, 1)$, not only can we form a wavelet basis for $L^2(\Lambda^\infty, M)$ by starting with paths of shape $J$, but we can use the data of the $(1, \ldots, 1)$-wavelets as the foundation for the $J$-shape wavelets. 

\begin{example}
Here we consider the example introduced in Example~\ref{ex:k-graph} (and denoted by $\Lambda_3$ there)  and compute some wavelets in this case.
 The corresponding 2-colored graph is given as the following;
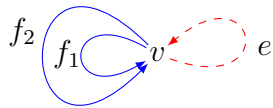
\begin{equation}
\begin{tikzpicture}[scale=1.5]
\node[inner sep=0.5pt, circle] (v) at (0,0) {$v$};
\draw[-latex, blue] (v) edge [out=150, in=210, loop, min distance=25, looseness=2.5] (v);
\draw[-latex, blue] (v) edge [out=135, in=225, loop, min distance=50, looseness=2.5] (v);
\draw[-latex, red, dashed] (v) edge [out=-20, in=40, loop, min distance=30, looseness=2.5] (v);
\node at (-0.8, 0) {$f_1$}; \node at (-1.2, 0.2) {$f_2$}; \node at (0.95, 0.05) {$e$};
\end{tikzpicture}
\end{equation}

\noindent and our factorization rules are:
\begin{align*}
{{f_2}}{{e}}={{e}{f_1}}\;\;\text{and}\;\; {{f_1}{e}}={{e}{f_2}}
\end{align*}
By these factorization rules, we see that %up to sets of $M$-measure $0,$
any particular infinite path in $x\in \Lambda^{\infty}$ can be chosen to be of the form
$$ef_{i_1}ef_{i_2}ef_{i_3}\cdots .$$
Setting ``color 1'' to be red and dashed, and ``color 2'' to be blue and solid, 
the two incidence matrices % corresponding to the two colors of edges and one vertex 
of this $2$-graph are $1\times 1$ and we have $(A_1)=(1),\;(A_2)=(2).$
Therefore the Perron Frobenius-measure on cylinder sets is:
$$M(Z(e))=1,\;M(Z(ef_i))=1/2,\; M(Z(ef_i e)) = 1/2, \; M(Z(ef_i e f_j)) = 1/4, \; \text{etc},$$
where $i, j \in \{1,2\}$.

Using Theorem 3.5 of \cite{FGKP}, we construct isometries $S_e,\; S_{f_1},$ and $S_{f_2}$ on $L^2(\Lambda^{\infty},M)$ satisfying
$$S_e^{\ast}S_e\;=\;S_{f_1}^{\ast}S_{f_1}\;=\;S_{f_2}^{\ast}S_{f_2}={I},$$
$$S_{e}S_e^{\ast}=S_{f_1}S_{f_1}^{\ast}+S_{f_2}S_{f_2}^{\ast}={I}.$$
and finally
$$S_eS_{f_1}\;=\;S_{f_2}S_e\;\;\text{and}\;\; S_eS_{f_2}=S_{f_1}S_e.$$
 Fix $\xi\in L^2(\Lambda^{\infty},M)$ and $x\equiv ef_{i_1}ef_{i_2}ef_{i_3}\cdots$, where $i_j\in \{1,2\}$.
Note that our factorization rules imply that $x=f_{i_1+1}ef_{i_2+1}ef_{i_3+1}\dots$, where the addition in the subscript of $f$ is taken modulo 2.

We define
$$S_e(\xi)(x)=\chi_{Z(e)}(x)1^{1/2}2^{0/2}\xi(\sigma^{(1,0)}x)\;=\;\xi(ef_{i_1+1}ef_{i_2+1}ef_{i_3+1}\cdots);$$
$$S_{f_1}(\xi)(x)=\chi_{Z(f_1)}(x)1^{0/2}2^{1/2}\xi(\sigma^{(0,1)}x)=\;2^{1/2}\chi_{Z(f_1)}(x)\xi(ef_{i_2+1}ef_{i_3+1}\cdots);$$
$$S_{f_2}(\xi)(x)=\chi_{Z(f_2)}(x)1^{0/2}2^{1/2}\xi(\sigma^{(0,1)}x)=2^{1/2}\chi_{Z(f_2)}(x)\xi(ef_{i_2+1}ef_{i_3+1}\cdots)).$$

We further calculate:
$$S_e^{\ast}(\xi)(x)\;=\;\chi_{Z(v)}(x)1^{-1/2}2^{0/2}\xi(ex)=\xi(ef_{i_1+1}ef_{i_2+1}ef_{i_3+1}\cdots);$$
$$S_{f_1}^{\ast}(\xi)(x)\;=\;2^{-1/2}\xi(f_1x)\;=\;2^{-1/2}\xi(ef_2ef_{i_1+1}ef_{i_2+1}ef_{i_3+1}\cdots);$$
$$S_{f_2}^{\ast}(\xi)(x)\;=\;2^{-1/2}\xi(f_2x)\;=\;2^{-1/2}\xi(ef_1ef_{i_1+1}ef_{i_2+1}ef_{i_3+1}\cdots).$$

One can easily verify that the partial isometries satisfy the appropriate commutation relations.

We now construct wavelets for this example, using the method of Theorem \ref{Wavelets-Theo}. Recall $M$ is the Perron-Frobenius measure on $\Lambda^{\infty},$ and define $\phi$  to be the constant function $1$ on $\Lambda^{\infty}.$ Take $$(j_1,j_2)=(1,1),$$
and let
$$\psi = \chi_{Z(ef_1)}-\chi_{Z(ef_2)}.$$
By using the main theorem of this section or direct calculation we verify that
$$\{\phi\}\cup\, \bigcup_{j=0}^{\infty}\{S_{\lambda}(\psi): \lambda \in \Lambda, d(\lambda) = (j,j)\}$$
is an orthonormal basis for $L^2(\Lambda^{\infty},M).$

\end{example}

\begin{example}
\label{ex:ledrappier}
In this example we describe how to construct the wavelets of this section for the Ledrappier 2-graph introduced in \cite{PRW-family}.

The skeleton of this 2-graph is 
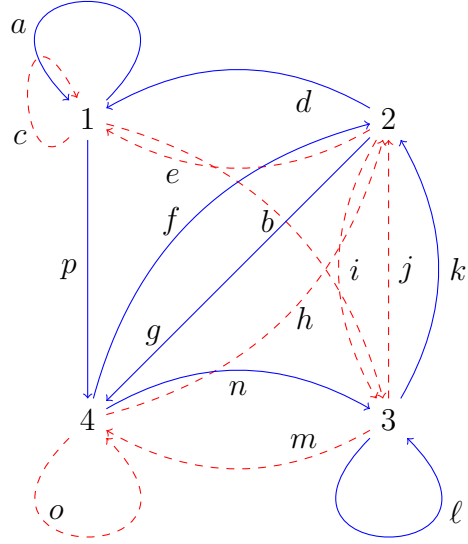
\[\begin{tikzpicture}[scale=2]
\node (v4) at (0,0){$4$};
\node (v3) at (2,0){$3$};
\node (v2) at (2,2){$2$};
\node (v1) at (0,2){$1$};
\draw[dashed, ->, red] (v4) to[bend right] node[right]{{\color{black} $h$}}  (v2);
\draw[dashed, ->, red] (v2) to[bend right] node[right]{{\color{black} $i$}} (v3);
\draw[dashed, ->, red] (v3) to node[right]{{\color{black} $j$}}  (v2);
\draw[dashed, ->, red] (v2) to[bend left] node[near end, left, below]{{\color{black} $e$}}  (v1);
\draw[dashed, ->, red] (v4) .. node[right, near start]{{\color{black} $o$}} controls (-1,-1) and (1,-1) ..   (v4);
\draw[dashed, ->, red] (v1) .. node[near start, left]{{\color{black} $c$}} controls (-0.5, 1.5) and (-0.5, 3) .. (v1);
\draw[dashed, ->, red] (v3) to[bend left] node[right, near start, above]{{\color{black} $m$}} (v4);
\draw[dashed, ->, red] (v1) to[bend left] node[left]{\color{black} $b$}  (v3);
\draw[->,blue] (v1) .. node[left, near end]{\color{black} $a$} controls (1, 3) and (-1, 3) .. (v1);
\draw[->, blue] (v1) to node[left]{\color{black} $p$}  (v4);
\draw[->, blue] (v4) to[bend left] node[left]{\color{black} $f$} (v2);
\draw[->, blue] (v2) to node[left, near end]{\color{black} $g$}  (v4);
\draw[->, blue] (v2) to[bend right] node[left, near start, below]{\color{black} $d$}  (v1);
\draw[->, blue] (v3) to[bend right] node[right]{\color{black} $k$}  (v2);
\draw[->,blue] (v3) .. node[right, near end]{\color{black} $\ell$} controls (1,-1) and (3, -1) .. (v3);
\draw[->,blue] (v4) to[bend left] node[right, below]{\color{black} $n$} (v3);
\end{tikzpicture}\]

If we define ``color 1'' to be blue and solid, and ``color 2'' red and dashed, the adjacency matrices are 
\[ A_1 = \begin{bmatrix}
1&  0 & 0 & 1 \\ 1 & 0 & 0 & 1 \\ 0 & 1 & 1& 0 \\ 0 &1 & 1 & 0 \end{bmatrix} 
\qquad 
A_2 = \begin{bmatrix}
1& 0 & 1 & 0 \\ 1 & 0 & 1 & 0 \\ 0 & 1 & 0 & 1 \\ 0 & 1 & 0 & 1 \end{bmatrix}\]

%Note that, again, the red edges are dashed and the solid edges are blue.  
Thus, there is a unique choice of factorization rules, since for each blue-red path (of length 2) between vertices $v$ and $w$, there is exactly one  red-blue path of length 2 between $v$ and $w$.

For this 2-graph, one can check that $\rho(A_1) = \rho(A_2) = 2$ and that $x^\Lambda = \frac{1}{4} (1, 1, 1, 1)$.  Let $J = (1,2);$ then 
\[D^{J}_{v_1} = \{acc, ace, aej, aeh, dhm, dho, djb, dji\}.\]
Similarly, $d^J_{v_i} = 8 $ for all $i$, and the inner product \eqref{eq:4.9innerprod} is given by 
\[\langle \vec{x}, \vec{y} \rangle = \frac{1}{32} \sum_{j=1}^8 x_j \overline{y_j}.\]

Thus, for each $i$, an orthonormal basis $\{c^{m, v_i}\}_{m=1}^7$ for the orthogonal complement of $\vec{1} \in \C^{d_{v_i}^J}$ is given by 
\begin{align*}
c^{1, v_i} = (4, -4, 0,0,0,0,0,0)  \qquad c^{2, v_i} = (0, 0, 4, -4, 0, 0, 0, 0) & \qquad c^{3, v_i} = (0, 0, 0, 0, 4, -4, 0, 0) \\
 c^{4, v_i} = (0,0,0,0,0,0,4, -4) \qquad  c^{5, v_i} & = \sqrt{2} (2,2, -2, -2, 0,0,0,0) \\
  c^{6, v_i} = \sqrt{2} (0,0,0,0,2,2,-2,-2) \qquad c^{7, v_i} & = (2,2,2,2,-2,-2,-2,-2).
\end{align*}

We will not list all of the 28 functions in $\mathcal{W}_{0,J}$ associated to the vectors $\{c^{m, v_i}\}$; however, we observe that 
\[f^{1, v_1} = 4 \Theta_{acc} - 4 \Theta_{ace}; \qquad f^{4, v_1} = 4\Theta_{djb} - 4 \Theta_{dji}; \qquad f^{5, v_1} = 2\sqrt{2} (\Theta_{acc}+ \Theta_{ace} - \Theta_{aeh} -\Theta_{aej}).\]
\end{example}

% Elizabeth would rename this seciton "Wavelets on \ell^2(\Lambda^0)"
\section{{Traffic analysis} wavelets on $\ell^2(\Lambda^0)$ {for a finite strongly connected $k$-graph $\Lambda$, and wavelets from spectral graph theory }}

Crovella and Kolaczyk argue in \cite{crovella-kolaczyk} that many crucial problems facing network engineers can profitably be approached using wavelets that reflect the structure of the underlying graph.  They give axioms that such \textbf{graph wavelets} must satisfy and provide some examples; Marcolli and Paolucci use  semibranching function systems to construct another example of graph wavelets in \cite{MP}.

We begin this section by showing  how to construct, from a $\Lambda$-semibranching function system, a family of wavelets on a higher-rank graph $\Lambda$ which meets   the specifications given in Section IV.A of \cite{crovella-kolaczyk}.  In other words, our wavelets $g^{m,J}$ of Section \ref{sec:traffic} are orthonormal functions supported on the vertices $\Lambda^0$ of the $k$-graph $\Lambda$, which have finite support and zero integral.  We thus hope that these wavelets will be of use for spatial traffic analysis on $k$-graphs, or, more generally, on networks with $k$ different types of links.

In a complementary perspective to the graph wavelets discussed in \cite{crovella-kolaczyk}, Hammond, Vandergheynst and Gribonval use the graph Laplacian in \cite{hammond} to construct wavelets on graphs.  
%In \cite{hammond}, Hammond, Vandergheynst,  and Gribonval propose a construction of wavelets on graphs which arises from the Laplacian of the graph.  
We show in Section \ref{sec:laplacian} how to extend their construction to higher-rank graphs, and we compare the wavelets thus constructed with the wavelets from Section \ref{sec:wavelets} and Section \ref{sec:traffic}.

\subsection{Wavelets for spatial traffic analysis}
\label{sec:traffic}
%{\color{red} While the wavelets $\{g^{m,v,n}\}_{m,n}$ span $L^2(\Lambda^0, \nu)$, we do not necessarily have $ g^{m,v,n} \perp g^{m', v, n'}$ if $n \not= n'$.  In contrast, the wavelets $\{g^{m,n}\}_{m,n}$ that we construct below satisfy $\langle g^{m,n}, g^{m', n'}\langle = \delta_{m,m'} \delta_{n, n'}$, but $\{g^{m,n}\}_{m,n}$ does not necessarily span $L^2(\Lambda^0, \tilde{\nu})$. } We illustrate these differences with the example of the Ledrappier 2-graph.

Suppose that $\Lambda$ is a finite strongly connected $k$-graph.
Fix $v \in \Lambda^0$ once and for all; for every vertex $w \in \Lambda$, fix a ``preferred path'' $\lambda_w \in v\Lambda w$.  We will use the Perron-Frobenius eigenvector $x^\Lambda$ of $\Lambda$, and the vector $\rho(\Lambda) \in (0, \infty)^k$ of eigenvalues of the adjacency matrices $A_i$, to construct our traffic analysis wavelets.

For each $ J \in \N^k$,  let
\[D_{J} = \{\lambda \in v\Lambda: d(\lambda) = J \text{ and } \lambda = \lambda_{s(\lambda)}\}.\]
Observe that $D_J$ might be empty.  We will assume that we can (and have) chosen our preferred paths $\lambda_w$ so that, for at least one $J \in \N^k$, $|D_J| \geq 2$.

If $|D_J| \geq 2$, define an inner product on $\C^{D_{J}}$ by
\begin{equation}
\label{inner-prod-traffic-2}
\langle \vec{v}, \vec{w} \rangle = \sum_{\lambda \in D_{J}}  \overline{v_\lambda} w_\lambda \rho(\Lambda)^{-J} x^\Lambda_{s(\lambda)}\end{equation}
and let $\{( c^{m,J}_\lambda)_{\lambda \in D_J}\}_{m=1}^{|D_J| - 1}$ be an orthonormal basis for the orthogonal complement of $(1, \ldots, 1) \in \C^{D_{J}}$ with respect to this inner product.

Define a  measure $\tilde{\nu}$ on $\Lambda^0$ by a variation on counting measure: if $E \subseteq \Lambda^0$, set
\[\tilde{\nu}(E) = \sum_{w \in E} \rho(\Lambda)^{-d(\lambda_w)} x^\Lambda_w.\]
For each  $(m,J)$ with $J \in \N^k$, $|D_J| >1$, and $m \leq |D_J| - 1$, define $g^{m,J} \in L^2(\Lambda^0, \tilde{\nu})$ by
\[g^{m,J}(w) = \left\{ \begin{array}{cl}
0 , & d(\lambda_w) \not= J \\
c^{m,J}_{\lambda_w}, & d(\lambda_w) = J
\end{array}\right. \]
Since the vectors $c^{m,J}$ are orthogonal to $(1, \ldots, 1)$ in the inner product \eqref{inner-prod-traffic-2}, we have
\begin{align*}
\int_{\Lambda^0} g^{m,J} \, d\tilde{\nu} & =  \sum_{w \in \Lambda^0} g^{m,J}(w) \tilde{\nu}(w) \\
& = \sum_{w: \lambda_w \in D_{J}} c^{m,J}_{\lambda_w} \rho(\Lambda)^{-J} x^\Lambda_w\\
&= 0
\end{align*}
for each $(m,J)$.  Moreover, if $g^{m,J}(w) \overline{g^{m', J'}(w)} \not= 0$, we must have $d(\lambda_w) = J = J'$; it follows that
\begin{align*}
\int_{\Lambda^0} g^{m,J} \overline{g^{m',J'}} \, d\tilde{\nu}
&= \delta_{J, J'} \sum_{w: \lambda_w\in D_{J}} c^{m,J}_{\lambda_w} \overline{c^{m',J}_{\lambda_w}} \rho(\Lambda)^{-J} x^\Lambda_{w} \\
&=   \delta_{J, J'}\delta_{m,m'}
\end{align*}
since the vectors $\{c^{m,J}\}$ form an orthonormal set with respect to the inner product \eqref{inner-prod-traffic-2}.  %{\color{red} It remains to check that $\{g^{m,n}\}$ spans $L^2(\Lambda^0, \nu)$.  However, I'm not sure this is true even for the Marcolli-Paolucci wavelets, nor does it seem that Crovella-Kolaczyk require it.}

In other words, %each function $g^{m,J}$ is supported on the set $S_{J}$ of vertices that are ``distance'' $m$ from $v$; and 
$\{g^{m,J}\}_{m,J}$ is an orthonormal set in $L^2(\Lambda^0, \tilde{\nu})$.  However, we observe that the wavelets $g^{m,J}$ will  not span $L^2(\Lambda^0, \tilde{\nu})$; at most, we will have $|\Lambda^0| - 1$ vectors $g^{m,J}$, which occurs when all the preferred paths $\lambda_w$ are in the same $D_J$. In this case, $\{g^{m,J}\}_m \cup \{f\}$ is an orthonormal basis for $L^2(\Lambda^0, \tilde{\nu})$, where $f$ is the constant function %{\color{blue} (Sooran: Does this function have norm 1?)}
\[f(w) = \frac{1}{\sqrt{\tilde{\nu}(\Lambda^0)}} = ({\rho(\Lambda)^J})^{1/2}.\]

As an example, we consider the  the Ledrappier 2-graph of Example \ref{ex:ledrappier}.  Define $v := v_1$ and observe that every vertex $v_i$ admits two paths $\lambda_i \in v \Lambda^{(1,2)} v_i$, so we can choose one of these for our ``preferred paths'' $\lambda_{v_i} := \lambda_i$.  In this case, $g^{m, J} = 0$ unless $J=(1,2)$; if we set 
\[c^{1, (1,2)} = (4,-4, 0,0), \qquad c^{2,(1,2)} = (0,0, 4, -4), \qquad c^{3,(1,2)} = \sqrt{2} (2,2, -2, -2),\]
then the vectors $\{c^{m, (1,2)}\}_m$ form an orthonormal basis for the orthogonal complement of $(1,1,1,1)$ with respect to the inner product 
\[\langle \vec{x}, \vec{y} \rangle = \sum_{i=1}^4 x_{\lambda_i} \overline{y_{\lambda_i}} \rho(\Lambda)^{(-1, -2)} x^\Lambda_{v_i} = \frac{1}{32} \sum_{i=1}^4 x_{\lambda_i} \overline{y_{\lambda_i}}.\]
Thus, our wavelets $g^{m, (1,2)}$ are given by
\begin{align*}
g^{1,(1,2)}(w) = \left\{ \begin{array}{cl} 4, & w = v_1 \\
										 -4, & w = v_2 \\
										 0, & \text{ else.}
\end{array} \right. & \quad 
g^{2, (1,2)}(w) = \left\{ \begin{array}{cl} 4, & w = v_3 \\
										 -4, & w = v_4 \\
										 0, & \text{ else.} \end{array} \right. 
\quad g^{3,(1,2)}(w) = \left\{ \begin{array}{cl} 2 \sqrt{2}, & w = v_1 \text{ or } v_2\\
										 -2\sqrt{2}, & w = v_3 \text{ or } v_4 . \end{array} \right.										 
\end{align*}
Since $|\Lambda^0| = 4$ and all of the functions $g^{m, (1,2)}$ are orthogonal (in $L^2(\Lambda^0, \tilde{\nu})$) to each other and to the constant function $f(w) = ({\rho(\Lambda)^{(1,2)}})^{1/2} = 2\sqrt{2}$, the set $\{g^{m, (1,2)}\}_m \cup \{f\}$ is an  orthonormal basis for $ L^2(\Lambda^0, \tilde{\nu})$.

\subsection{Wavelets on $\ell^2(\Lambda^0)$ coming  from spectral graph theory}
\label{sec:laplacian}
In this section we extend the definition of the graph Laplacian given by Hammond, Vandergheynst, and Gribonval in \cite{hammond} to define a Laplacian for higher-rank graphs.  For a graph (or $k$-graph) on $N$ vertices, the (higher-rank) graph Laplacian is an $N\times N$ positive definite matrix.  While the construction of the higher-rank graph Laplacian, given in Definition \ref{def-k-graph-Laplacian} below,
differs slightly from that of the graph Laplacian of \cite{hammond}, the two matrices share many of the same structural properties.  Consequently, the majority of the results from \cite{hammond} apply to the higher-rank graph Laplacian as well, with nearly verbatim proofs.  Thus, we include very few proofs in this section, instead referring the reader to \cite{hammond}.

There are many definitions of the graph Laplacian in the literature (cf.~\cite{biggs, chung-cheeger, jor-pear-0}); using the graph Laplacian to construct wavelets is also common.  We plan to explore the relationships among these varying constructions in future work.

Our definition of the $k$-graph Laplacian more closely parallels those of \cite{biggs, chung-cheeger} than that of \cite{hammond}, because the latter requires that the vertex matrix of the graph be symmetric.  While this is always the case for an undirected graph, it is rarely the case for a $k$-graph, so we have chosen to define the $k$-graph Laplacian following the lines indicated in \cite{biggs, chung-cheeger}.  We observe that in the case when the vertex matrices are indeed symmetric, the definitions in \cite{hammond} and \cite{biggs} of the graph Laplacian coincide.
%
%Moreover,  in our constructions below all of the  weights on the edges of the graph are implicitly  taken to be one.
%For (weighted) discrete  Laplacians on graphs and their spectral analysis see the work of  Jorgensen and Pearse  \cite{jor-pear-0}, \cite{jor-pear}.

\begin{defn} (see \cite[Definition 4.2]{biggs}, \cite{chung-cheeger})
 \label{def-k-graph-Laplacian}
Let $\Lambda$ be a finite $k$-graph with $N= |\Lambda^0|$ vertices.    For each $1 \leq s \leq k$, let $N_1^s = |\Lambda^{e_s}|$ be the number of edges of color $s$.  Define the  incidence matrix $M_s= (m^s_{i,j})_{i=1,...,N; j=1,...,N_1^s}$,  where
\[
m^s_{i,j}:= \left\{ \begin{array}{cc} +1 & \hbox{if } r(e_j) \not= s(e_j) \text{ and } r(e_j) = v_i  \\  -1 & \hbox{if }r(e_j) \not= s(e_j) \text{ and } s(e_j) = v_i  \\ 0 & \hbox{otherwise} \end{array}         \right.
\]
We then define the Laplacian $\Delta_{\Lambda}$ of $\Lambda$ to be
\[
\Delta_{\Lambda} := \sum_{s=1}^k M_sM_s^T.
\]
\end{defn}

\begin{rmk}
When $k=1$ and both definitions apply,
Proposition 4.8 of \cite{biggs} tells us that Definition \ref{def-k-graph-Laplacian} agrees with the definition of the graph Laplacian given in \cite{hammond}.

Furthermore,  each summand $M_s M_s^T$ is a positive definite symmetric matrix;  %{\color{red} Precise reference from Biggs?}; 
it follows (cf.~\cite{chung}) that $\Delta_\Lambda$ has an orthonormal basis of eigenvectors and that the eigenvalues of $\Delta_\Lambda$ are all non-negative.
\end{rmk}

 Hammond, Vandergheynst,  and Gribonval point out in \cite{hammond} that the graph wavelets they describe can be viewed as arising from the graph Laplacian in the same way that continuous wavelets arise from the one-dimensional Laplacian operator $d/dx^2$.  The set of functions $\{e^{i \omega x}:\;\omega\in\mathbb R\}$ used to define the Fourier transform on $\R$ are also eigenfunctions of the Laplacian  $d/dx^2$; thus, one could interpret the inverse Fourier transform 
 \[
f(x) = \frac1{2\pi} \int \hat{f}(\omega) e^{i \omega x} d\omega
\]
as providing the coefficients of $f$ with respect to the eigenfunctions of the Laplacian.
We define the \textbf{higher-rank graph Fourier transform} analogously.

To be precise, let $\{\vec{v}_i\}_{i=1}^{N}$ be a basis of eigenvectors for $\Delta_\Lambda$.

Henceforth, we assume that %$\Lambda$ is connected, and that 
we have ordered the eigenvalues $\lambda_1, \ldots, \lambda_{N}$  such that
\[  \lambda_1 \leq  \lambda_2 \leq \lambda_3 \cdots \leq  \lambda_{N} .\]
The \textbf{higher-rank graph Fourier transform} of a function $f \in C(\Lambda^0)$ is the function $\hat{f} \in C(\Lambda^0)$ given by
\[\hat{f}(\ell) = \langle \vec{v}_\ell, f \rangle = \sum_{n=1}^{N_0} {\vec{v}_\ell(n)} f(n).\]

The motivation for the following definition comes from the calculations in Section 5.2 of \cite{hammond}.  Specific choices for wavelet kernels, and motivations for these choices, can be found in Section 8 of the same article.
\begin{defn}
\label{def:spectral-wavelet-oper}
Let $\Lambda$ be a finite $k$-graph.
A \textbf{wavelet kernel} is a function $g: \R \to \R$ such that
\begin{enumerate}
\item $g$ is $(M+1)$-times continuously differentiable for some $M \in \N$, and $g^{(M)}(0) =: C \not= 0$;
\item On a neighborhood of 0, $g$ is \lq well approximated' (as in Lemma 5.4 of \cite{hammond}) by $c x^M$, where $c = C/M!$; 
\item $\int_0^\infty \frac{g^2(x)}{x} \, dx =: C_g < \infty$.
\end{enumerate}

Given a wavelet kernel $g$,  the \textbf{$k$-graph wavelet operator} $T_g = g(\Delta_\Lambda)$ acts on $f \in C(\Lambda^0)$ by
\[T_g(f)(m) = \sum_{\ell=1}^{N} g(\lambda_\ell) \hat{f}(\ell) \vec{v}_\ell(m) = \sum_{\ell, n =1}^{N} g(\lambda_\ell) {\vec{v_\ell}(n)} \vec{v}_\ell(m) f(n).\]
For any $t \in \R$ we also have a \textbf{time scaling} $T_g^t$ given by
\[T_g^t(f) = g(t\Delta_\Lambda)(f) = m \mapsto \sum_{\ell=1}^{N} g(t \lambda_\ell) \hat{f}(\ell) \vec{v}_\ell(m).\]

For each  $k$-graph wavelet operator $T_g$ and each $t \in \R$ we obtain a family $\{\psi_{g, t, n}\}_{1 \leq n \leq N}$ of \textbf{higher-rank graph wavelets}: If $\delta_n \in C(\Lambda^0)$ is the indicator function at the $n$th vertex of $\Lambda$,
\[ \psi_{g,t,n} := T_g^t \delta_n = m \mapsto \sum_{\ell=1}^{N} g(t \lambda_\ell) {\vec{v_\ell}(n) } \vec{v_\ell}(m).\] % give precise formula?
\end{defn}
%
%\begin{rmk}
%Observe that $\psi_{g,t,n}$ is orthogonal to $\vec{v}_1$ by construction, since the eigenvector $\vec{v}_1$ has eigenvalue zero.
%\end{rmk}

\begin{prop}\cite[Lemma 5.1]{hammond}
Suppose $g: \R \to \R$ is a wavelet kernel and $g(0) = 0$.  Then  every function $f \in C(\Lambda^0)$ can be reconstructed from $\{\psi_{g,t,n}\}_{t,n}$:
\[f =  \frac{1}{C_g} \sum_{n=1}^{N} \int_0^\infty \frac{\langle \psi_{g, t, n}, f \rangle}{t} \psi_{g,t,n} \, dt.\]
% check: do I really need to separate out n=1 and larger n cases?
\end{prop}
\begin{proof}
Recall that 
\[\langle \psi_{g,t,n}, f\rangle = \sum_{\ell=1}^{N} {\psi_{g,t,n}(\ell)} f(\ell) = \sum_{\ell, m=1}^{N} {g(t\lambda_m)} \vec{v_m}(\ell) {\vec{v_m}(n)} f(\ell)\]
since the eigenvectors $\vec{v_m}$ are real-valued.
Thus, %{\color{blue} (Sooran thinks that the following should have $\sum_{n=1}^{N_0}$ in the first two formulas).}
\[ \sum_{n=1}^{N} \langle \psi_{g,t,n}, f\rangle \psi_{g,t,n}(k) = \sum_{j,\ell,m, n= 1}^{N} f(\ell) g(t\lambda_m) g(t\lambda_j) \vec{v_m}(\ell) \vec{v_m}(n) \vec{v_j}(n) \vec{v_j}(k) = \sum_{\ell, m = 1}^{N} f(\ell) g(t\lambda_m)^2 \vec{v_m}(\ell) \vec{v_m}(k)\]
since the orthonormality of the eigenvectors $\{\vec{v_m}\}_m$ implies that 
\[\langle \vec{v_m}, \vec{v_j} \rangle = \sum_n \vec{v_m}(n) \vec{v_j}(n) = \delta_{m,j}.\]
It follows that %{\color{blue} (Sooran : the first formula should contain $\sum_{n=1}^{N_0}$).}
\begin{align*}
\sum_{n=1}^{N}\int_0^\infty \frac{\langle \psi_{g,t,n}, f \rangle}{t} \psi_{g,t,n} (k)\, dt &= \sum_{\ell, m=1}^{N} f(\ell) \vec{v_m}(\ell) \vec{v_m}(k) \int_0^\infty \frac{g(t\lambda_m)^2}{t} \, dt \\
&= \sum_m \hat{f}(m) \vec{v_m}(k)\int_0^\infty \frac{g(t\lambda_m)^2}{t} \, dt = \sum_m \hat{f}(m) \vec{v_m}(k) \int_0^\infty \frac{g(u)^2}{u} \, du\\
&= \sum_m \hat{f}(m) \vec{v_m}(k) C_g.
\end{align*}
The symmetry of the Fourier transform implies that $f(k) = \sum_m \hat{f}(m) \vec{v_m}(k)$, which finishes the proof.
\end{proof}

Our hypothesis that the wavelet kernel $g$ be well approximated by $c x^M$ for some $M \in \N$ ensures that the wavelet $\psi_{g,t,n}$ is nearly zero on vertices more than $M$ steps away from $n$.  In other words, the wavelets $\psi_{g,t,n}$ are localized near the vertex $n$.  The proof of this result is identical to that given in \cite{hammond} for the case $k=1$.

\begin{prop}\cite[Theorem 5.5]{hammond}
If $d(m, n) > M$, and if there exists $t' \in \R$ such that $|g^{(M+1)}(x) |$ is uniformly bounded for $x \in [0, t'\lambda_M]$, then there exist constants $D, t''$ such that for all $t < \min\{ t', t''\},$
\[\frac{\psi_{g,t,n}(m)}{\| \psi_{g,t,n}\|} \leq Dt.\]
\end{prop}

\begin{example}
We now construct spectral $k$-graph wavelets for the Ledrappier 2-graph of Example \ref{ex:ledrappier}.  Ordering the edges alphabetically, and assigning ``color 1'' to the blue, solid edges and ``color 2'' to the red, dashed edges, we obtain 
\[ M_1 = \begin{bmatrix}
0 & 1 & 0 & 0 & 0 & 0 & 0 & -1 \\ 
0 & -1 & 1 & -1 & 1 & 0 & 0 & 0 \\
 0 & 0 & 0& 0& -1 & 0 & 1 & 0 \\
  0 & 0 & -1 & 1 & 0 & 0 & -1 & 1
\end{bmatrix}
\qquad 
M_2 = \begin{bmatrix}
-1 & 0 & 1 & 0 & 0 & 0 & 0 & 0\\
0 & 0 & -1 & 1 & -1 & 1 & 0 & 0 \\
1 & 0 & 0 & 0 & 1 & -1 & -1& 0 \\
0 & 0 & 0& -1 & 0 & 0& 1 & 0
\end{bmatrix}\]
Thus, 
\[\Delta_\Lambda = M_1 M_1^T + M_2 M_2^T = \begin{bmatrix}
 2 & -1 & 0 & -1 \\ -1 & 4 & -1 & -2 \\ 0 & -1 & 2 & -1 \\ -1 & -2 & -1 & 4 
\end{bmatrix}
+ \begin{bmatrix}
2 & -1 & -1 & 0 \\ -1 & 4 & -2 & -1 \\ -1 & -2 & 4 & -1 \\ 0 & -1 & -1 & 2
\end{bmatrix} 
= \begin{bmatrix}
4 & -2 & -1 & -1 \\ -2 & 8 & -3 & -3 \\ -1 & -3 & 6 & -2 \\ -1 & -3 & -2 & 6
\end{bmatrix}.\]
Computing the eigenvalues and eigenvectors (to two decimal places of accuracy), we obtain 
\[\lambda_1 = 0, \lambda_2 = 5.17, \lambda_3 = 10.83, \lambda_4 = 8\]
and 
\[\vec{v_1} = (1,1,1,1), \quad \vec{v_2} =   ( -0.85,  0.15, 0.35, 0.35), \]
\[ \vec{v_3} =  ( -.15, 0.85, -0.35, -0.35), \quad \vec{v_4} = (0, 0, -0.71, 0.71 ).\]
Then the wavelets $\psi_{g,t,n}$ in $\ell^2(\Lambda^0)$ are given by 
\[\psi_{g,t,n}(m) = \sum_{\ell = 1}^4 g(t \lambda_\ell) \vec{v_\ell}(n) \vec{v_\ell}(m).\]
As in \cite{hammond}, one possible wavelet kernel (with $N=2$) is 
\[g(x) := \left\{ \begin{array}{cl}
x^2, & 0 \leq x \leq 1 \\
-5 + 11x -6x^2 + x^3, & 1 < x < 2 \\
4x^{-2}, & x \geq 2
\end{array}\right.\]
Observe that $g(x) > 0 \ \forall \ x > 0$.

To distinguish these wavelets $\psi_{g,t,n}$ from those of Section \ref{sec:traffic}, we observe that (for fixed  $t \in \R$) each of the four wavelets $\psi_{g,t,n}$ is supported on all four vertices of the Ledrappier 2-graph.
\end{example}

\end{document}